\documentclass{amsart}
\usepackage{amssymb}
\usepackage{amsmath}

\usepackage{colortbl}

\newtheorem{theorem}{Theorem}[section]
\newtheorem{lemma}[theorem]{Lemma}
\newtheorem{corollary}[theorem]{Corollary}
\theoremstyle{definition}
\newtheorem{definition}[theorem]{Definition}

\theoremstyle{remark}
\newtheorem{remark}[theorem]{Remark}

\numberwithin{equation}{section}

\newcommand{ \mint }{ {\int\hspace{-0.38cm}- }}

\newcommand{ \mr }{ \mathbb{R} }

\newcommand{ \ba }{ \mathbf{a}}

\begin{document}

\title[Quasi-linear equations of Schr\"dinger type]
{Interior and boundary $W^{1,q}$-estimates for quasi-linear elliptic equations of Schr\"odinger type}

%    Information for first author

\author{Mikyoung Lee}
  %  Address of record for the research reported here
  \address{Department of Mathematics, Pusan National University, Busan
46241, Republic of Korea}
%    Current address
\email{mikyounglee@pusan.ac.kr}

\author{Jihoon Ok}
%    Address of record for the research reported here
\address{Department of Applied Mathematics and Institute of Natural Science, Kyung Hee University, Yongin 17104, Republic of Korea}
%    Current address
\email{jihoonok@khu.ac.kr}

%    \thanks will become a 1st page footnote.
%\thanks{This work was supported by }

%    General info
\subjclass[2010]{Primary 35J10, 35J92; Secondary 35J25, 35B65}

%\date{\today}

%\dedicatory{This paper is dedicated to our advisors.}

\keywords{Schr\"odinger operator; $p$-Laplacian; gradient estimate}

\begin{abstract} We consider nonlinear elliptic equations that are naturally obtained from the elliptic Schr\"odinger equation $-\Delta u +Vu=0$ in the setting of the calculus of variations, and obtain $L^q$-estimates for the gradient of weak solutions. In particular, we generalize  a result of  Shen in  [Ann. Inst. Fourier 45 (1995), no. 2, 513--546] in the nonlinear setting by using a different approach. This allows us to consider discontinuous coefficients with a small BMO semi-norm and non-smooth boundaries which might not be Lipschitz continuous.
\end{abstract}

\maketitle

%\section*{This is an unnumbered first-level section head}
%This is an example of an unnumbered first-level heading.

%% The correct journal style for \specialsection is all uppercase; a known bug
%% in amsart.cls prevents this, so input must be uppercase until it is fixed.
%\specialsection*{This is a Special Section Head}
%\specialsection*{THIS IS A SPECIAL SECTION HEAD}
%This is an example of a special section head%
%%%%%%%%%%%%%%%%%%%%%%%%%%%%%%%%%%%%%%%%%%%%%%%%%%%%%%%%%%%%%%%%%%%%%%%%
%\footnote{Here is an example of a footnote. Notice that this footnote
%text is running on so that it can stand as an example of how a footnote
%with separate paragraphs should be written.
%\par
%And here is the beginning of the second paragraph.}%
%%%%%%%%%%%%%%%%%%%%%%%%%%%%%%%%%%%%%%%%%%%%%%%%%%%%%%%%%%%%%%%%%%%%%%%%

\section{\bf Introduction} \label{secintro}
The present paper is devoted to the study of interior and boundary $L^q$-integrability for the gradient of weak solutions to time independent  quasi-linear equations  of the $p$-Schr\"odinger type
\begin{equation}\label{maineqm}
-\mathrm{div}\, (|Du|^{p-2}Du) + V |u|^{p-2}u  =  0 \ \  \textrm{ in } \  \Omega,
\end{equation}
where $1<p<\infty$, $\Omega\subset \mr^n$($n\geq2$) is open and bounded, and the non-negative potential $V$ is taken in an appropriate class. We notice that if $p=2,$ the equation \eqref{maineqm} becomes
\begin{equation} \label{propeq}
- \Delta u  + Vu=0\ \ \ \textrm{in}\ \ \Omega,
\end{equation}
which is the classical (elliptic) Schr\"odinger equation. In the viewpoint of the calculus of variations, the equation \eqref{maineqm} is the Euler-Lagrange equation of the following functional
$$
W^{1,p}(\Omega)\ni u \ \ \mapsto\ \ \int_{\Omega} \left[|Du|^p+V|u|^p\right]\, dx,
$$
hence it is one of nonlinear generalizations of the  Schr\"odinger equation \eqref{propeq} in a natural way. Moreover, problems of this type raise in various areas of physics, such as nonlinear quantum field theory, nonlinear optics, plasma physics, condensed matter physics, biophysics, fluid mechanics, etc. We refer to \cite{APT1,BS1, LS1,MF1,SS} for the general physical background of this equation.

Research on the Schr\"odinger type equations which are fundamental ones of quantum mechanics plays a significant role in the fields of mathematical physics. In particular, $L^q$-regularity theory for linear Schr\"odinger equations was first introduced by Shen \cite{Sh1}. He obtained $L^q$-estimates by assuming that $V$ belongs to the $\mathcal B_{\gamma}$ class for some $\gamma\geq \frac n2$ which is a certain reverse H\"older class (see below for the definition of $\mathcal B_\gamma$). More precisely, for the Schr\"odinger equations with non-divergence data of the form $-\Delta u+Vu=f$ in $\mr^n$, he showed  $\Vert D^2u\Vert_{L^q(\mr^n)}+\Vert Vu\Vert_{L^q(\mr^n)} \leq c(q)\|f\|_{L^q(\mr^n)}$  for all $1<q\leq \gamma$, and  for the equations with divergence data of the form
\begin{equation}\label{divschrodinger}
-\Delta u+Vu=-\textrm{div}\, F\ \ \ \text{in}\ \ \mr^n,
\end{equation}
he also did
$$
\Vert Du \Vert_{L^{q}(\mr^n)} +\chi_{\{q\leq 2\gamma\}} \Vert V^{\frac12}u \Vert_{L^{q}(\mr^n)} \leq c(q)\Vert F \Vert_{L^{q}(\mr^n)} , \ \ \textrm{for all}\ \  (\gamma^*)' \leq q \leq \gamma^*,
$$
where $\gamma^*=\frac{n\gamma}{n-\gamma}$ when $\gamma<n$ (if $\gamma\geq n$, then $q$ can be any number in $(1,\infty)$). Here, we remark that the range of $q$ is optimal, see \cite[Section 7]{Sh1}.  These results have been recently extended to linear elliptic/parabolic Schr\"odinger equations with discontinuous coefficients on sufficiently smooth domains in several papers for instance \cite{BHS1,BBHV1,PT1}, by using the results in \cite{Sh1} together with the commutator method and the standard flattening and covering arguments. We also refer to \cite{CFG1, D1, FPR, K1, PT1, Sh0, Sh1}  for the regularity theory for (elliptic) Schr\"odinger equations.

The general aim of this paper is to establish interior and boundary $L^q$-regularity theory for nonlinear Schr\"odinger equations in non-smooth domains. In particular, as mentioned earlier, we deal with quasi-linear equations of $p$-Laplacian type which are the natural generalizations of the classical Schr\"odinger equation in the divergence setting. Moreover, the domains we consider here might be non-graph domains which are
beyond the class of Lipschitz domains. We point out that the approach used in \cite{Sh1} cannot be applied to the nonlinear setting. Indeed, Shen  in \cite{Sh1} derived the decay estimates for the fundamental solution by means of the Fefferman-Phong Lemma in \cite{Fe} by introducing an auxiliary function $m(x,V)$ which is well-defined for $q \geq \frac{n}{2}.$ Furthermore, on the boundary region we cannot make use of the flattening argument since our domain is supposed to be non-smooth. Therefore, an alternative approach must be adopted in order to handle the structures of the nonlinear operators and the non-smooth domains. In our best knowledge, the present paper is a new one treating $L^q$-estimates for Schr\"odinger equations in a non-linear setting and even for linear Schr\"odinger equations on non-smooth domains.

Now let us present our main equations. We are concerned with the Dirichlet problem for the quasi-linear Schr\"odinger equation of the form
\begin{equation}
\label{maineq}
\left\{\begin{array}{rclcc}
-\mathrm{div}\, \mathbf{a}(x,Du) + V |u|^{p-2}u&  =  & -\mathrm{div}\, (|F|^{p-2}F) & \textrm{ in } & \Omega,  \\
u & = & 0 & \textrm{ on } & \partial \Omega,
\end{array}\right.
\end{equation}
where $1<p<\infty$, $\Omega$ is open and bounded in $\mr^{n}$ with $n\geq2$, and $V:\Omega\to \mr$ is non-negative and at least satisfies $V\in L^{n/p}(\Omega)$ if $p<n$ and $V\in L^t(\Omega)$ for some $t>1$ if $p\geq n$.
 A  given vector valued function $ \mathbf{a} : \mr^n \times \mr^n \rightarrow \mr^n $ is a Carath\'eodory function, that is,  $ \mathbf{a}$ is measurable in the $x$-variable and differentiable in the $\xi$-variable. We will always assume that $\mathbf{a}$ satisfies the following growth and ellipticity conditions:
\begin{equation}\label{aas1}
| \mathbf{a}(x,\xi)|+ | D_{\xi}  \mathbf{a}(x,\xi)||\xi| \leq L |\xi|^{p-1}
\end{equation}
and
\begin{equation}\label{aas2}
D_{\xi}  \mathbf{a}(x,\xi)\, \eta \cdot  \eta \geq \nu |\xi|^{p-2}|\eta|^2
\end{equation}
for almost all $x \in \mr^n$ and any $\xi, \eta \in \mr^n$ and for some constants $L, \nu$ with $0< \nu \leq 1 \leq L.$  A prime example of the nonlinearlity $\mathbf{a}$ is
$$
\mathbf{a}(x,\xi) = a(x)|\xi|^{p-2}\xi,\ \ \nu\leq a(\cdot)\leq L,
$$
which is the $p$-Laplacian with the coefficient $a(\cdot).$  We also remark that the above condition \eqref{aas2} implies the monotonicity condition:
\begin{equation}\label{mono}
\left( \mathbf{a}(x,\xi) - \mathbf{a}(x,\eta) \right) \cdot (\xi-\eta) \geq c(p,\nu) \left( |\xi|^2 + |\eta|^2 \right)^{\frac{p-2}{2}} |\xi-\eta|^2
\end{equation}
for any $\xi, \eta \in \mr^n$ and a.e. $x \in \mr^n.$
In particular, if $p \geq 2,$ it can be the following
\begin{equation}\label{mono1}
\left( \mathbf{a}(x,\xi) - \mathbf{a}(x,\eta) \right) \cdot (\xi-\eta) \geq c(p,\nu)  |\xi-\eta|^p.
\end{equation}
Under the above basic setting,  we say that  $u \in W^{1,p}_0(\Omega)$ is a  weak solution to the problem \eqref{maineq} if
\begin{equation}\label{weakform}
 \int_{\Omega} \mathbf{a}(x, Du) \cdot D\varphi \, dx  + \int_{\Omega} V |u|^{p-2}u \cdot \varphi  \, dx =\int_{\Omega} |F|^{p-2} F \cdot D \varphi\, dx
\end{equation}
holds for any $\varphi \in W_0^{1,p}(\Omega).$ We note that if $u\in W^{1,p}_0(\Omega)$, $\|Du\|_{L^p(\Omega)}$ and $\|Du\|_{L^p(\Omega)}+\|V^{\frac1p}u\|_{L^p(\Omega)}$ are equivalent  by the condition of the potential $V$ and Sobolev-Poincar\'e's inequality, and that the existence and the uniqueness of the weak solution of \eqref{maineq} (even in the case of a non-zero Dirichlet boundary condition such that  $u=g$ on $\Omega$ with $g\in W^{1,p}(\Omega)$) follow from the theory of nonlinear functional analysis, see for instance \cite[Chapter 2]{Sho1}.

For the potential $V:\Omega\to \mr$ considered in the problem \eqref{maineq}, we suppose that
$V$ belongs to $\mathcal{B}_{\gamma}$ for some $ \gamma \in [ \frac{n}{p}, n)$ when $p<n$ and for some $\gamma \in (1,n)$ when $p\geq n.$ We say that $V:\mr^n\to [0,\infty)$ belongs to $\mathcal{B}_\gamma$ for some $ \gamma >1$ if $V\in L^{\gamma}_{loc}(\mr^n)$ and there exists a constant $b_{\gamma}>0$ such that the \textit{reverse H\"older inequality}
\begin{equation}\label{VBqclass}
\left( \frac{1}{|B|} \int_{B} V^{\gamma} \, dx \right)^{\frac{1}{\gamma}} \leq b_\gamma \left( \frac{1}{|B|} \int_{B} V \, dx \right)
\end{equation}
holds for every ball $B$ in $\mr^n.$  This  $\mathcal{B}_\gamma$ class which is a wide class including all nonnegative polynomials was introduced independently by Muckenhoupt \cite{Mu} and Gehring \cite{Ge} in the study of weighted norm inequalities and quasi-conformal mapping, respectively. One notable example of this element is $ V(x) = |x|^{-n/ \gamma} $
which actually belongs to the $\mathcal{B}_{\tilde{\gamma}}$ class for all $\tilde{\gamma} < \gamma.$ Moreover, the $B_\gamma$ class is strongly connected to the Muckenhoupt class, for which we will discuss later in Section \ref{Preliminaries}.

%%%%%%%%%%%%%%%%%%%%%%%

Our main result is the global integrability of $Du$ and also $V^{\frac1p}u$ for the weak solutions $u$ to the problem  \eqref{maineq} with respect to the one of $F$, under a suitable discontinuity condition on  the nonlinearity $\mathbf{a}$ and a minimal structure condition on the boundary of the domain $\Omega$ that will be described later in Definition \ref{smallbmo} and \ref{Defreifenberg}, respectively.  More precisely,  we prove that
\begin{equation}\label{implication1}
F\in L^q\  \Longrightarrow  \ Du \in L^q\ \ \text{for each } \left\{\begin{array}{cl}
q\in[p,\gamma^*(p-1))&  \text{when }\gamma\in [\frac np,n),\\
q\in[p,\infty)&  \text{when } \gamma\in[n,\infty),
\end{array}\right.
\end{equation}
\begin{equation}\label{implication2}
F\in L^q \ \ \Longrightarrow \ \ V^{\frac1p}u \in L^q \ \ \ \text{for each } p\leq q \leq p \gamma,
\end{equation}
by obtaining relevant estimates, see Corollary  \ref{maincor} and Remark \ref{mainrmk} in the next section. We would like to emphasize that for the Schr\"odinger equation \eqref{divschrodinger}, that is, the equation  \eqref{maineq} when $p=2$ and  $\ba(x,\xi)\equiv \xi,$  our results cover the ones in \cite[Corollary 0.10]{Sh1} for $q\geq p=2$. Note that, in this linear case, the validity of the implications \eqref{implication1} and \eqref{implication2} for $ \gamma^*<q<2$ can be achieved via the duality argument, see for instance \cite{Um1}.

For the equation \eqref{maineq} with the null potential, i.e., $V\equiv0$, the $L^q$-estimates, which is sometimes called the (nonlinear) Calder\'on-Zygmund estimates, have been widely studied by many authors.
 Iwaniec \cite{Iw1} first obtained the $L^q$-estimates for the $p$-Laplace equations with $p\geq 2$, and then DiBenedetto \& Manfredi \cite{DM1} extended his result to the $p$-Laplace systems with $1<p<\infty$.  Later, Caffarelli \& Peral \cite{CP1} considered general equations of the $p$-Laplacian type with discontinuous nonlinearities. Furthermore, Acerbi \& Mingione generalized $L^q$-estimates for the parabolic $p$-Laplace systems with discontinuous coefficients \cite{AM1}. We also refer to \cite{BR1,LO1,MP1,KZ1,Mis1} for problems with $p$-Laplacian type and \cite{AM0,BO1,BOR1,CM1,Ok1} for problems with nonstandard growth.

 %%%%%%%%%%%%%%%%%%%%%%%%
We briefly discuss  the outline of the proof of the $L^q$-estimates.  As mentioned earlier, our approach is different from the one used in \cite{Sh1} which is  based on the linear operator theory. We adopt a perturbation argument which has turned out to be very useful for the study on the regularity theory for linear and nonlinear PDEs. In particular, we employ the method introduced by Acerbi \& Mingione in \cite{AM1}, see also \cite{Min1} for its origin. To be more concrete, we apply an exit time argument to a nonlinear functional of $Du, V^{\frac1p}|u|$ and $F,$ in order to construct a suitable family of balls which covers the level set for $|Du|+V^{\frac1p}|u|$.  Then, on each ball, we compare our equation \eqref{maineq} with the homogeneous equation
$$
 -\mathrm{div}\, \ba(x,Dw)+V|w|^{p-2}w=0.
$$
The main part at this step is to find the maximal integrability of $Dw$ and $V^{\frac1p}w$ with corresponding estimates.  In view of the classical regularity theory we know the  $L^\infty$-boundedness of $w$ (see Lemma~\ref{supvlem}), from which together with the result in our recent paper \cite{LO1} (see Theorem~\ref{thmDwbdd}), we see that $Dw \in L^{\gamma^*(p-1)}$ and $V^{\frac1p}w\in L^{p\gamma}$ (see Lemma~\ref{lem42}). Here, we point out that the corresponding estimates  \eqref{DwDwVwest} and \eqref{DwDwVwest1} are derived in a very delicate way. Especially, at this stage, the $\mathcal B_\gamma$ condition of $V$ plays a crucial role, so that we take advantage of the idea of Fefferman \& Phong in \cite{Fe} to obtain the modified version of Fefferman-Phong Lemma (see Lemma~\ref{lemrVbddpq}). Then from those corresponding estimates, the $L^q$-estimates for $|Du|+V^{\frac1p}|u|$ is derived by the comparison argument when $q\leq p\gamma$. Furthermore, applying the results in \cite{LO1}, we eventually obtain the $L^q$-estimates for $|Du|$ when $p\gamma<q\leq \gamma^*(p-1).$

The remainder of this paper is organized as follows. In the next section, we state our main results with primary assumptions imposed on the nonlinearlity $\mathbf{a}$ and the domain $\Omega.$ Section~\ref{Preliminaries} deals with the basic properties of $\mathcal{B}_{\gamma}$ class and the auxiliary lemmas to prove the main results. In Section~\ref{sechomo}, we show higher integrability of $Du$ and $V^{\frac1p}u$ for weak solutions $u$ to localized equations of our main problem \eqref{maineq} with $F\equiv 0.$
In Section~\ref{secComestimates}, we obtain the comparison estimates, and finally prove
main results, Theorem~\ref{mainthm} and Corollary \ref{maincor}, in Section~\ref{sec gradient estimates}.

%The rest of these notes is divided into four sections and an appendix. We first give basic results on the Cauchy problem in Section 2. In Section 3, we study the existence of standing waves. Section 4 is devoted to stability whereas Section 5 deals with instability.

%%%%%%%%%%%%%%%%%%%%%%%%%

%%%%%%%%%%%%%%%%%%%%%%%%%%%

\section{\bf  Main result }
\label{secpre}
We start this section with standard notation and definitions. We denote the open ball $\mr^n$ with center $y\in \mr^n$ and radius $r>0$ by $B_r(y)= \{ x \in \mr^{n} :  |x-y|< r \}.$ We also denote $\Omega_r (y)= B_r(y) \cap \Omega$ and $ \partial_w\Omega_{r}(y) = B_r(y) \cap \partial \Omega.$
For the sake of simplicity, we write $B_r=B_r(0),$ $B_r^+=B_r^+(0)$ and $\Omega_{r} = \Omega_{r}(0).$
We shall use the notation
$$ \mint_{U} g \; dx := \frac{1}{|U|} \int_{U} g \;dx. $$

The following two definitions are associated with the main assumptions imposed on the nonlinearlity  $\mathbf{a}$ and the domain $\Omega.$
\begin{definition} \label{smallbmo} We say that $\mathbf{a}=\mathbf{a}(x,\xi)$ is \textit{$(\delta,R)$-vanishing} if
$$
\sup_{0<\rho\leq R}  \ \sup_{y\in\mathbb{R}^n} \mint_{B_{\rho}(y) }  \left|\Theta\left( \mathbf{a},B_{\rho}(y) \right)(x) \right| \, dx   \leq \delta,
$$
where $$ \Theta\left( \mathbf{a},B_{\rho}(y) \right)(x):= \sup_{\xi \in \mr^n \setminus \{ 0\} } \frac{  \left|\mathbf{a}(x,\xi)-\overline{\mathbf{a}}_{B_{\rho}(y)}(\xi)\right|}{|\xi|^{p-1} }$$
and $$ \overline{\mathbf{a}}_{B_{\rho}(y)}(\xi) := \mint_{B_{\rho}(y)} \mathbf{a}(x,\xi) \;dx.$$
\end{definition}

The above definition implies that the map $x\mapsto \ba(x,\xi)/|\xi|^{-p}$ is a (locally) BMO function with the BMO semi-norm less than or equal to $\delta$ for all $\xi\in\mr^n$. Hence we see that the nonlinearity  $\ba$ can be discontinuous for the $x$-variable. In particular, if $\ba(x,\xi)=a(x)|\xi|^{p-2}\xi$, then this definition means that $a(\cdot)$ is a BMO function.

\begin{definition}\label{Defreifenberg} Given $\delta \in (0,\frac18)$ and $R>0,$ we say that $\Omega$ is a $(\delta, R)$-Reifenberg flat domain if for every $x \in \partial\Omega$ and every $\rho \in (0, R],$ there exists a coordinate system $\{ y_1, y_2, \dots, y_n\}$ which may depend on $\rho$ and $x,$ such that in this coordinate system $x=0$ and that
$$ B_{\rho}(0) \cap \{ y_n > \delta \rho \} \subset B_{\rho}(0)  \cap \Omega \subset B_{\rho}(0) \cap \{ y_n > -\delta \rho \}.
$$
\end{definition}
In the above definition of the Reifenberg flat domain, $\delta$ is usually supposed to be less than $\frac18$. This number comes from the Sobolev embedding, see for instance \cite{To1}. However, it is not important since we will consider $\delta$ sufficiently small. We note that the Lipschitz domains with the Lipschitz constant less than or equal to $\delta$ belong to the class of $(\delta,R)$-Reifenberg flat domains for some $R>0$. In addition, we remark that the $(\delta,R)$-Reifenberg flat domain $\Omega$ has the following measure density conditions:
\begin{equation}\label{dencon}
\sup_{0<\rho\leq R} \sup_{y \in \overline{\Omega}} \frac{\left|B_{\rho}(y)\right|}{\left|\Omega \cap B_{\rho}(y)\right|} \leq \left( \frac{2}{1-\delta} \right)^{n} \leq \left( \frac{16}{7} \right)^n,
\end{equation}
\begin{equation}\label{dencon1}
\inf_{0<\rho\leq R} \inf_{y \in \partial\Omega} \frac{|\Omega\cap B_{\rho}(y) |}{\left|B_{\rho}(y)\right|} \geq \left( \frac{7}{16} \right)^n.
\end{equation}
We refer to \cite{BW1,PS1,Re1,To1} for more details on the Reifenberg flat domains and their applications.

Now let us state the main results in this paper.
\begin{theorem}
\label{mainthm}
Let $u\in W^{1,p}_0(\Omega)$ be a weak solution to \eqref{maineq}. Suppose that $V \in \mathcal{B}_{\gamma}$ for some $ \gamma \in [ \frac{n}{p}, n)$ when $p<n$ and for some $\gamma \in (1,n)$ when $p\geq n.$ For $p\leq q <\gamma^*(p-1)$, there exists  a small  $\delta = \delta(n, p, L, \nu)>0$ so that if $\mathbf{a}$ is $( \delta, R)$-vanishing and $\Omega$ is a $( \delta, R)$-Reifenberg flat domain for some $R\in(0,1),$ then we have for any $x_0\in \overline{\Omega}$ and $r \in (0, \frac{R}{4}]$ satisfying 
%$$ (4r)^{p\gamma - n} \int_{\Omega_{4r}(x_0)} V^{\gamma} \;dx \leq 1, $$
$ (4r)^{p - \frac{n}{\gamma}} \Vert V \Vert_{L^{\gamma}(\Omega_{4r}(x_0))} \leq 1,$
\begin{eqnarray}\label{mainlocalest}
\nonumber\left( \mint_{\Omega_{ r}(x_0)} |Du|^q +\chi_{\{q<p\gamma\}} \left[V^{\frac1p}|u|\right]^{q}\,  dx\right)^{\frac1q} &&\\
&&\hspace{-6cm}\leq c \left( \mint_{\Omega_{4r}(x_0)} |Du|^p + \left[ V^{\frac1p} |u|\right]^p \,dx \right)^{\frac{1}{p}}+ c  \left(\mint_{\Omega_{4 r}(x_0)}   \left|F \right|^{q} \, dx\right)^{\frac1q}
\end{eqnarray}
for some $c=c(n,p,q,\gamma,L,\nu,b_{\gamma})>0,$ where $\chi_{\{q<p\gamma\}}:= 1$ if $q<p\gamma$ and $\chi_{\{q<p\gamma\}}:= 0$ if $q\geq p\gamma.$
\end{theorem}

\begin{remark}
Let $\Omega$ be a $(\delta,R)$ Reifenberg flat domain for some small $\delta>0$ and $R>0$ and $V\in \mathcal B_\gamma$ with $\gamma\geq \frac{n}{p}$ and $p>1$. Define 
$$
\rho(y,V):= \sup\left\{r\in(0,R]:r^{p - \frac{n}{\gamma}} \Vert V \Vert_{L^{\gamma}(\Omega_{r}(y))} \leq 1 \right\},\quad y\in \overline{\Omega}.
$$
Then by H\"older's inequality, the $\mathcal B_\gamma$ condition of $V$ and \eqref{dencon}, we see that the function $\rho(y,V)$ is comparable to
$$
\tilde\rho(y,V):= \sup\left\{r\in(0,R]:\frac{1}{r^{n-p}} \int_{\Omega_{r}(y)} V \,dx \leq 1 \right\},\quad y\in \overline{\Omega},
$$
i.e. $\frac{1}{c}\tilde\rho(y,V) \leq \rho(y,V)\leq c \tilde\rho(y,V)$ for all $y\in\overline\Omega$ with constant $c$ independent of $y$. When $p=2$, recalling the function $m(y,V)$ defined in \cite[Definition 1.3]{Sh1}, we notice that $\tilde \rho(y,V)$ is a local version of $\frac{1}{m(y,V)}$. In view of this observation, it seems that the restriction $(4r)^{p - \frac{n}{\gamma}} \Vert V \Vert_{L^{\gamma}(\Omega_{4r}(x_0))} \leq 1$ in Theorem~\ref{mainthm} is reasonable.
\end{remark}

\begin{remark}
In Theorem~\ref{mainthm}, we can obtain the estimate \eqref{mainlocalest} uniformly with respect to $x_0$ by taking $r>0$ such that
$$
r \leq \frac{1}{4}\min\left\{R,\Vert V \Vert_{L^{\gamma}(\Omega)}^{-\frac{\gamma}{p\gamma-n}}\right\},
$$
since this together with the fact that $p\gamma>n$ implies  
$$ (4r)^{p - \frac{n}{\gamma}} \Vert V \Vert_{L^{\gamma}(\Omega_{4r}(x_0))} \leq  (4r)^{p-\frac{n}{\gamma}} \Vert V \Vert_{L^{\gamma}(\Omega)} \leq 1.
$$
\end{remark}

As a consequence of Theorem~\ref{mainthm} and the preceding remark, we obtain the global gradient estimates for solutions to \eqref{maineq}. 
\begin{corollary}\label{maincor}
Let $u\in W^{1,p}_0(\Omega)$ be a weak solution to \eqref{maineq}. Suppose that $V \in \mathcal{B}_{\gamma}$ for some $ \gamma \in [ \frac{n}{p}, n)$ when $p<n$ and for some $\gamma \in (1,n)$ when $p\geq n.$ For $p\leq q <\gamma^*(p-1)$, there exists a small $ \delta = \delta(n, p, L, \nu) \in (0,\frac18)$ so that if $\mathbf{a}$ is $( \delta, R)$-vanishing and $\Omega$ is a $( \delta, R)$-Reifenberg flat domain for some $R\in(0,1),$ then we have
\begin{equation}\label{maingloest}
\Vert Du \Vert_{L^{q}(\Omega)} +\chi_{\{q<p\gamma\}} \Vert V^{\frac1p}u \Vert_{L^{q}(\Omega)} \leq c\left(\frac{\mathrm{diam}(\Omega)}{\tilde{R}}\right)^{n\left(\frac{1}{q}-\frac{1}{p}\right)}\Vert F \Vert_{L^{q}(\Omega)} 
\end{equation}
for some $c=c(n,p,q,\gamma,L,\nu,b_{\gamma})>0,$ where
$\tilde{R} := \min\left\{ R, \Vert V \Vert_{L^{\gamma}(\Omega)}^{-\frac{1}{p-\frac{n}{\gamma}}}\right\}.$
Here,
 $\chi_{\{q<p\gamma\}}:= 1$ if $q<p\gamma$ and $\chi_{\{q<p\gamma\}}:= 0$ if $q\geq p\gamma.$
\end{corollary}

\begin{remark}\label{mainrmk}  If $V\in \mathcal B_{\gamma}$, then $V$ belongs to the $\mathcal B_{\gamma+\epsilon}$ class for some small $\epsilon>0$ by virtue of the self improving property of the $\mathcal B_\gamma$ class in Lemma \ref{lemself} below. Therefore, by considering $\gamma+\epsilon$ instead of $\gamma$ in Theorem \ref{mainthm} and Corollary \ref{maincor},  the range of $q$ can be extended to $p\leq q\leq \gamma^*(p-1)$, and the estimates \eqref{mainlocalest} and \eqref{maingloest} can be replaced by
\begin{eqnarray*}
\left( \mint_{\Omega_{ r}(x_0)} |Du|^q +\chi_{\{q\leq p\gamma\}} \left[V^{\frac1p}|u|\right]^{q}\,  dx\right)^{\frac1q} &&\\
&&\hspace{-6cm}\leq c \left( \mint_{\Omega_{4r}(x_0)} |Du|^p + \left[ V^{\frac1p} |u|\right]^p \,dx \right)^{\frac{1}{p}}+ c  \left(\mint_{\Omega_{4 r}(x_0)}   \left|F \right|^{q} \, dx\right)^{\frac1q}
\end{eqnarray*}
and
$$
\Vert Du \Vert_{L^{q}(\Omega)} +\chi_{\{q\leq p\gamma\}} \Vert V^{\frac1p}u \Vert_{L^{q}(\Omega)} \leq c\,\Vert F \Vert_{L^{q}(\Omega)},
$$
respectively.

Finally, if the map $x\mapsto \ba(x,\xi)|\xi|^{-(p-1)}$ is in VMO uniformly for the $\xi$-variable, that is,
$$
\lim_{\rho\to0}\left(\sup_{y\in\mathbb{R}^n} \mint_{B_{\rho}(y) }  \left|\Theta\left( \mathbf{a},B_{\rho}(y) \right)(x) \right| \, dx\right)=0,
$$
and the boundary of $\Omega$ is $C^1$, we have the implications \eqref{implication1} and \eqref{implication2} for every $q$ in the ranges stated in there.
\end{remark}

\begin{remark} Under the assumption that $
V\in \mathcal B_{\gamma}\ \  \text{for some }\gamma\in[n,\infty),
$ in stead of $\gamma\in(\frac{n}{p},n),$  we see that the results of Theorem \ref{mainthm} and Corollary \ref{maincor} hold for any $q\in[p,\infty).$ Indeed, if $
V\in \mathcal B_{\gamma}$ for some $\gamma\in[n,\infty),
$
it is easily seen that  $V\in \mathcal B_{\gamma'}$ for any $\gamma'\in (1,\gamma)$ with the constant $b_{\gamma'}=b_{\gamma}$, by the definition of the $\mathcal B_{\gamma}$ class. Then for any $q\in[p,\infty),$ choosing $\gamma' = \gamma' (n,p,q) \in (\frac{n}{p},n)$ such that
$$
q<(\gamma')^*(p-1),
$$ we consequently obtain the results of Theorem \ref{mainthm} and Corollary \ref{maincor}  for any $q\in[p,\infty).$ Hence, we have the implications \eqref{implication1} for $\gamma\in[n,\infty)$ and \eqref{implication2}.
\end{remark}

\section{Preliminaries}\label{Preliminaries}

\subsection{$\mathcal{B}_\gamma$ class}\ \\
  In order to introduce primary features of the $\mathcal{B}_\gamma$ class,   let us first recall the Muckenhoupt $A_p$ and $A_\infty$ classes. We say that nonnegative function $V\in L^1_{loc}(\mr^n)$ is in  the $A_p$ class, $V\in A_p$,  for some $1\leq p<\infty$ if and only if
$$
\sup_B \left(\mint_BV\, dx\right)\left(\mint_B V^{-\frac{1}{p-1}}\, dx\right)^{p-1}<\infty
$$
and that $V\in L^1_{loc}(\mr^n)$ is in  the $A_\infty$ class, $V\in A_\infty$, if and only if
$$
\sup_B \left(\mint_BV\, dx\right)\exp\left(\mint_B \log V^{-1}\,dx\right)<\infty,
$$
where the supremum is taken over all balls $B \subset \mr^n.$  From the definition of $\mathcal{B}_{\gamma}$ in \eqref{VBqclass}, we notice that $V\in \mathcal{B}_{\gamma}$  for $\gamma\in(1,\infty)$ if and only if
$$
\sup_B\left(\mint_BV\, dx\right)^{-1} \left(\mint_{B}V^\gamma\,dx\right)^{\frac1\gamma}<\infty,
$$
where the supremum is taken over all balls $B \subset \mr^n,$ which is very similar to the condition of $A_p$, or $A_\infty$, class. Indeed, we have the following equivalent condition. For its proof, we refer to \cite[Theorem 9.3.3]{G1}.

\begin{lemma}\label{lemequiv}Let $V\in L^1_{loc}(\mr^n)$ be nonnegative. The following are equivalent
\begin{itemize}
\item[(1)] $V\in A_\infty$.
\item[(2)] There exist $\theta,\sigma\in(0,1)$ such that
$$
\left|\left\{x\in B: V(x)\leq \theta \mint_{B}V\, dy \right\}\right|\leq \sigma|B|
$$
for every ball $B$ in $\mr^n$.
\item[(3)] $V\in B_{\gamma}$ for some $\gamma>1$.
\item[(4)] $V\in A_p$ form some $p>1$.
\end{itemize}

In particular, if $V\in A_\infty$, then there exists $\theta\in(0,1)$ such that
$$
\left|\left\{x\in B: V(x)\leq \theta \mint_{B}V\, dy \right\}\right|\leq \frac12|B|
$$
for every ball $B$ in $\mr^n$, that is, one can choose that $\sigma=\frac12$.
\end{lemma}

From the above equivalent conditions and the self improving property of the $A_p$ classes, one can deduce the self improving property of the $\mathcal{B}_\gamma$ classes as follows.

\begin{lemma}\label{lemself}
If $V\in \mathcal{B}_\gamma$ for some $\gamma>1$, then $V\in B_{\gamma+\epsilon}$ for some small $\epsilon>0$.
\end{lemma}

\subsection{Gradient estimates for equations with mixed data}\ \\
The next two results are local Calder\'on-Zygmund estimates for elliptic equations of $p$-Laplace type involving mixed data. Let us consider the following problem
\begin{equation}
\label{homoeqfF}
\left\{\begin{array}{rclcc}
-\mathrm{div}\, \mathbf{a}(x,Dw)&  =  &f - \mathrm{div}\,(|F|^{p-2}F) & \textrm{ in } & \Omega_{2r}(x_0),  \\
w & = & 0 & \textrm{ on } & \partial_w\Omega_{2r}(x_0)\ \text{if}\ B_{2r}(x_0)\not\subset\Omega.
\end{array}\right.
\end{equation}
Here, the `mixed data' means $f-\mathrm{div}\,(|F|^{p-2}F)$. We note that if $f\equiv 0$, the Calder\'on-Zygmund estimates have been obtained in for instance \cite{BR1,MP1}, and if $F\equiv 0$ and $2-\frac1n <p<n$, these can be found in for instance \cite{Ph1}. From those results, we can expect a similar result for the mixed problem \eqref{homoeqfF}, and  the authors recently obtained the desired one in \cite{LO1}. By the Sobolev's embedding, we consider two cases that $q>\max\{p,\frac{(p-1)n}{n-1}\}$ with $1<p<\infty$ and $p<q\leq \frac{(p-1)n}{n-1}$ with $p>n$.

\begin{theorem}\label{thmDwbdd}
Let $1<p<\infty$ and $q>\max\{p,\frac{(p-1)n}{n-1}\}$. There exists a small $ \delta = \delta(n, L, \nu, p, q) \in (0,\frac18) $ so that if $\mathbf{a}$ is $( \delta, R)$-vanishing and $\Omega$ is a $(\delta,R)$-Reifenberg flat domain for some $R\in(0,1)$, then for any $x_0\in\overline\Omega$, $r\in(0,\frac{R}{2}]$ and weak solution $w\in W^{1,p}(\Omega_{2r}(x_0))$ of \eqref{homoeqfF} with $F\in L^q(\Omega_{2r}(x_0))$ and $f\in L^{(q/(p-1))_*}(\Omega_{2r}(x_0))$, we have

\begin{eqnarray}\label{homoeqestimate}
\nonumber \left(\mint_{\Omega_r(x_0)} |Du|^{q} \, dx\right)^{\frac{1}{q}}   
%&\leq& c \left( \mint_{\Omega_{2r}(x_0)} |Du|^{p}\, dx\right)^{\frac{1}{p}} \\
%\nonumber &&+ c \left(\mint_{\Omega_{2r}(x_0)} [\mathcal{M}_{p_0}(|f|^{p_0}\chi_{\Omega_{2r}(x_0)})]^{\frac{q}{p_0(p-1)}} + |F|^q \,dx \right)^{\frac1q}\\
\nonumber &\leq& c \left( \mint_{\Omega_{2r}(x_0)} |Du|^{p}\, dx\right)^{\frac{1}{p}}+c\left(\mint_{\Omega_{2r}(x_0)} |F|^q \,dx\right)^\frac{1}{q} \\
&&+ c \left(   \mint_{\Omega_{2r}(x_0)} |r f  |^{\left( \frac{q}{p-1}\right)_*} \, dx      \right)^{\frac{1}{\left( \frac{q}{p-1}\right)_*(p-1)}}
\end{eqnarray} 
for some $c=c(n, L,\nu,  p,q)>0.$
% where $p_0>1$ is denoted by
%$$
%p_0:=\left\{\begin{array}{lcl}
%(p^*)'=(p')_*=\frac{np}{np-n+p}&\textrm{if}&p\in(1,n),\\
%\textrm{any number in }  (1,(\frac{q}{p-1})_*) & \text{if} &p\in [n,\infty).
%\end{array}\right.
%$$
\end{theorem}

\begin{theorem}\label{thmDwbdd1}
Let $n<p<\infty$, $p<q\leq \frac{(p-1)n}{n-1}$ and $1<\tilde q <n$. There exists a small $\delta = \delta(n, L, \nu, p, q)\in (0,\frac18) $ so that if $\mathbf{a}$ is $( \delta, R)$-vanishing, $\Omega$ is a $(\delta,R)$-Reifenberg flat for some $R\in(0,1)$, then for any $x_0\in\overline\Omega$, $r\in(0,\frac{R}{2}]$ and for any weak solution $w\in W^{1,p}(\Omega_{2r}(x_0))$ of  \eqref{homoeqfF} with $F\in L^q(\Omega_{2r}(x_0))$ and $f\in L^{\tilde q}(\Omega_{2r}(x_0))$, we have 
\begin{eqnarray*}
\left(\mint_{\Omega_r(x_0)} |Dw|^{q} \, dx\right)^{\frac{1}{q}}   &\leq& c \left( \mint_{\Omega_{2r}(x_0)} |Dw|^{p}\, dx\right)^{\frac{1}{p}}+c\left(\mint_{\Omega_{2r}(x_0)} |F|^q \,dx\right)^\frac{1}{q} \\
&&+ c \left(   \mint_{\Omega_{2r}(x_0)} |r f  |^{\tilde q} \, dx      \right)^{\frac{1}{\tilde q(p-1)}}
\end{eqnarray*}
for some constant $c=c(n,  L,\nu, p,q, \tilde q)>0.$
\end{theorem}

\subsection{Auxiliary lemmas}\ \\
 We first recall the local boundedness (up to boundaries) for weak solutions to the equation \eqref{maineq} with $F\equiv 0$, which is a classical regularity result and we refer to \cite[Chapter 2.5]{LU} and \cite[Chapter 7]{Gi}. We point out that Reifenberg flat domains $\Omega$ considered in this paper have the measure density conditions \eqref{dencon} and \eqref{dencon1}, which are enough to obtain the boundedness for weak solutions.

\begin{lemma}\label{supvlem}
Let $1<p<\infty$ and suppose that the bounded domain $\Omega\subset\mr^n$ is $(\delta,R)$-Reifenberg flat for some $\delta\in(0,1/2)$ and $R>0$. Assume that $\mathbf{a}$ satisfies
\begin{equation}\label{aAss}
|\ba(x,\xi)|\leq L|\xi|^{p-1} \ \ \text{and}\ \ \ba(x,\xi)\cdot \xi\geq \nu |\xi|^p
\end{equation}
for any $x,\xi\in \mr^n$ and for some $0<\nu\leq L$, and that the nonnegative function $V$ satisfies $V \in L^{\gamma}(\Omega)$ for some $\gamma \in (\frac{n}{p}, n)$ when $p<n$ and for some $\gamma>1$ when $p\geq n.$  Then for any ball $B_{2r}(x_0)$ with $x_0\in\overline\Omega$ and $r\in(0,\frac{R}2]$ satisfying $ (2r)^{p - \frac{n}{\gamma}} \Vert V \Vert_{L^{\gamma}(\Omega_{2r}(x_0))} \leq 1,$ 
and  for any weak solution $w\in W^{1,p}(\Omega_{2r}(x_0))$ of
$$
\left\{\begin{array}{rclcc}
-\mathrm{div}\, \mathbf{a}(x,Dw) + V |w|^{p-2}w&  =  &0 & \textrm{ in } & \Omega_{2r}(x_0),  \\
w & = & 0 & \textrm{ on } &  \partial_w\Omega_{2r}(x_0)\ \text{if}\ B_{2r}(x_0)\not\subset\Omega,\end{array}\right.
$$
we have  that
\begin{equation*}
\|w\|_{L^\infty(\Omega_r (x_0))} \leq c \left( \mint_{\Omega_{2r}(x_0)} |w|^p \, dx \right)^{\frac{1}{p}}
\end{equation*}
for some constant $c = c(n, p, L, \nu,\gamma) >0.$
\end{lemma}
\begin{proof}
Let us define the rescaled maps 
$$ \tilde{\mathbf{a}}(x,\xi) = \mathbf{a}(rx, \xi),\ \tilde{w}(x) = \frac{w(rx)}{r}, \ \tilde{V}(x) = r^p V(rx), \text{ and } \tilde{\Omega} = \left\{ \frac{x}{r} : x \in \Omega  \right\}.$$
Then one can check that $ \tilde{\mathbf{a}}$ satisfies the assumption \eqref{aAss} with the same constants $L$ and $\nu$, $\tilde{\Omega}$ is $(\delta,\frac{R}{r})$-Reifenberg flat, $\tilde{V} \in L^{\gamma}(\tilde{\Omega})$,  and $\tilde{w} \in W^{1,p}(\Omega_{2}(x_0)) $ is a weak solution of 
$$
\left\{\begin{array}{rclcc}
-\mathrm{div}\, \tilde{\mathbf{a}}(x,D\tilde{w}) + \tilde{V} |\tilde{w}|^{p-2}\tilde{w}&  =  &0 & \textrm{ in } & \tilde{\Omega}_{2}(x_0),  \\
\tilde{w} & = & 0 & \textrm{ on } &  \partial_w\tilde{\Omega}_{2}(x_0)\ \text{if}\ B_{2}(x_0)\not\subset\tilde{\Omega}. \end{array}\right.
$$
By the classical local boundedness result (see, for instance, \cite[Chapter 2.5]{LU} and \cite[Chapter 7]{Gi}),
we see that 
\begin{equation*}
\|\tilde{w}\|_{L^\infty(\tilde{\Omega}_1 (x_0))} \leq c \left( \mint_{\tilde{\Omega}_{2}(x_0)} |\tilde{w}|^p \, dx \right)^{\frac{1}{p}}
\end{equation*}
for some constant $c = c(n, p, L, \nu,\gamma, \Vert \tilde{V} \Vert_{L^{\gamma}(\tilde{\Omega}_{2}(x_0))}) >0.$ Here, since the constant $c$ in the above estimate is increasing as a function of $\Vert \tilde{V} \Vert_{L^{\gamma}(\tilde{\Omega}_{2}(x_0))}$ and 
$$  \Vert \tilde{V} \Vert_{L^{\gamma}(\tilde{\Omega}_{2}(x_0))} \leq (2r)^{p - \frac{n}{\gamma}} \Vert V \Vert_{L^{\gamma}(\Omega_{2r}(x_0))} \leq 1,$$
$c$ can be replaced by a larger constant independent of $\Vert \tilde{V} \Vert_{L^{\gamma}(\tilde{\Omega}_{2}(x_0))}$.
Therefore, after scaling back, we can arrive at the desired bound of $w$.
\end{proof}

The following is the standard iteration lemma, whose proof can be found in for instance \cite{HL1}.
\begin{lemma}\label{teclem}
Let $ g :[a,b] \to \mr $ be a bounded nonnegative function. Suppose that for any $s_1,s_2$ with  $ 0< a \leq s_1 < s_2 \leq b $,
$$
g(s_1) \leq \tau g(s_2) + \frac{A}{(s_2-s_1)^{\beta}}+B
$$
where $A,B \geq 0, \beta >0$ and $0\leq \tau <1$. Then we have
$$
g(s_1) \leq c\left(  \frac{A}{(s_2-s_1)^{\beta}}+ B \right)
$$
for some constant $c=c(\beta, \tau) >0$.
\end{lemma}

We end this section by introducing a basic inequality which will be used later. Although its proof is elementary, we shall give it in detail  for the sake of readability.
\begin{lemma}\label{lemineq} For any function  $g\in W^{1,p}(B_r)$ with any $r>0,$ we have
$$
\frac{1}{r^{n+p}}\int_{B_r} \int_{B_r} \left| g(x)-g(y) \right|^p \,dx dy  \leq c \int_{B_r} \left|  Dg(x) \right|^p \,dx
$$
for some $c=c(n,p)>0$.
\end{lemma}
\begin{proof}
Without loss of generality, we shall assume that $g\in C^1(B_r).$  Using H\"older's inequality, Fubini's theorem and the fact that $|x-y|\leq 2r,$ we observe that
\begin{eqnarray*}
\int_{B_r} \int_{B_r} \left| g(x)-g(y) \right|^p \,dx dy &=& \int_{B_r} \int_{B_r} \left| \int_0^1Dg(t(x-y)+y)\cdot (x-y)\,dt \right|^p\, dxdy\\
&\leq & (2r)^p\int_0^1 \int_{B_r} \int_{B_r}  |Dg(t(x-y)+y)|^p\, dx dy dt.
\end{eqnarray*}
Here we point out that $t(x-y)+y\in B_r$ for any $x,y\in B_r$. Then we use a change of variable with $x=\tilde x+y$ and apply Fubini's theorem to obtain  that
\begin{eqnarray*}
\int_{B_r} \int_{B_r} \left| g(x)-g(y) \right|^p \,dx dy &\leq & (2r)^p\int_0^1 \int_{B_r} \int_{B_r(-y)}  |Dg(t\tilde x+y)|^p\, d\tilde x dy dt\\
&\leq & (2r)^p\int_0^1 \int_{B_{2r}} \int_{B_{r-\frac{|\tilde x|}{2}}(-\frac{\tilde x}{2})}  |Dg(t\tilde x+y)|^p\, dy d\tilde x dt.
\end{eqnarray*}
Note that $B_{r-\frac{|\tilde x|}{2}}(t\tilde x-\frac{\tilde x}{2})\subset B_r$ for any $\tilde x\in B_{2r}$ and any $t\in (0,1).$ Hence, by letting $\tilde y=t\tilde x+y$, we have
$$
\int_{B_{r-\frac{|\tilde x|}{2}}(-\frac{\tilde x}{2})}  |Dg(t\tilde x+y)|^p\, dy= \int_{B_{r-\frac{|\tilde x|}{2}}(t\tilde x-\frac{\tilde x}{2})}  |Dg(\tilde y)|^p\, d\tilde y \leq \int_{B_{r}}  |Dg(\tilde y)|^p\, d\tilde y,
$$
which implies that
$$
\int_{B_r} \int_{B_r} \left| g(x)-g(y) \right|^p \,dx\, dy \leq  (2r)^{n+p}|B_1|\int_{B_{r}}  |Dg(\tilde y)|^p\, d\tilde y.
$$
This completes the proof.
\end{proof}

\section{Gradient estimates for homogenous equations}\label{sechomo}

In this section we obtain gradient estimates for weak solutions to localized equations of \eqref{maineq} with $F\equiv 0$. Let us start with the following lemma, which is in fact a key lemma in  our proofs.

\begin{lemma}\label{lemrVbddpq}
Let $1 < p < \infty$ and suppose $V \in \mathcal{B}_{\gamma}$ for some $ \gamma \in [ \frac{n}{p}, n)$ when $p<n$ and for some $\gamma \in (1,n)$ when $p\geq n.$
Then for any function $w \in W^{1,p}(B_{r})$ with $0<r<1,$  we have
\begin{eqnarray}\label{prVbddpq}
\nonumber \left( \mint_{B_{r}}\left| \frac{w}{r}\right|^{p} \, dx\right)^{\frac1p} \left( \mint_{B_{r}} [r^pV]^{\gamma}\, dx\right)^{\frac{1}{p\gamma }} &\leq& c\max\left\{\left( \mint_{B_{r}} [r^pV]^{\gamma}\, dx\right)^{\frac{1}{p\gamma }},1\right\} \\
&&\hspace{-1.5cm}\times\,  \left[ \left( \mint_{B_{r}} \left|  Dw \right|^p \,dx\right)^{\frac1p} +   \left(\mint_{B_{r}} V\left|  w \right|^p  \,dx\right)^{\frac1p} \right]
\end{eqnarray}

for some constant $c=c\left(n,p, b_\gamma \right)>0.$
\end{lemma}

\begin{proof}
By Lemma \ref{lemineq}, we have
\begin{equation*}
\int_{B_r} \left|  Dw(x) \right|^p \,dx \geq \frac{c}{r^{n+p}}\int_{B_r} \int_{B_r} \left| w(x)-w(y) \right|^p \,dxdy.
\end{equation*}
Moreover, we also have
\begin{equation*}
\int_{B_r} V(x) \left|  w(x) \right|^p  \,dx = \frac{1}{r^{n}|B_1|}\int_{B_r} \int_{B_r} V(y) \left| w(y) \right|^p\,dxdy.
\end{equation*}
Then we have that for any constant $c_0>0$,
\begin{eqnarray}\label{pDwVwc0}
\nonumber&&\int_{B_r} \left|  Dw(x) \right|^p \,dx + \int_{B_r} V(x)\left|  w(x) \right|^p  \,dx \\
\nonumber&&\qquad \geq \frac{c}{\max\{c_0,1\}\,r^{n}} \bigg( \int_{B_r}\int_{B_r} \frac{c_0\left| w(x)-w(y) \right|^p}{r^p}\,dy dx \\
&& \hspace{4.5cm} + \int_{B_r} \int_{B_r} V(y)\left| w(y) \right|^p \,dydx\bigg).
\end{eqnarray}
Note that it is easily seen that
\begin{eqnarray*}
 &&\frac{c_0\left| w(x)-w(y) \right|^p}{r^p} + V(y)|w(y)|^p \\
 &&\quad \geq \min\left\{\frac{c_0}{r^p},V(y)\right\} \left(  \left| w(x)-w(y) \right|^p\ + |w(y)|^p \right) \geq \min\left\{\frac{c_0}{r^p},V(y)\right\} \frac{|w(x)|^p}{2^{p-1}} .
\end{eqnarray*}
Hence, inserting this into \eqref{pDwVwc0}, we obtain
\begin{eqnarray*}
&&\int_{B_r} \left|  Dw(x) \right|^p \,dx + \int_{B_r} \left|  w(x) \right|^p V(x) \,dx \\
&&\qquad \geq \frac{c}{\max\{ c_0,1\}\,r^n}  \int_{B_r} \left(\int_{B_r} \min_{y \in B_r}\left\{ \frac{c_0}{r^p} , V(y)\right\}  \,dy\right)\left| w(x)\right|^p\,dx.
\end{eqnarray*}

On the other hand, since $V \in \mathcal{B}_\gamma,$ by Lemma \ref{lemequiv} there exists $\theta>0$ such that
$$\left| \left\{ x \in B : V(x) \geq \theta \mint_{B} V(y)\,dy \right\} \right| \geq  \frac12  |B|$$
for every ball $B \subset \mr^n.$ Then we take $$c_0:= \theta\, r^p  \mint_{B_r} V(y) dy$$ to discover that
$$ \int_{B_r} \min_{y \in B_r}\left\{ \frac{c_0}{r^p} , V(y)\right\} dy   \geq \frac{c_0}{2r^p} |B_r| = \frac{c_0\, r^{n-p}|B_1|}{2}.  $$
Therefore we get
\begin{eqnarray*}
\frac{c_0r^{-p}}{\max\{ c_0,1\}}  \int_{B_r} |w|^p\,dx &\leq& \frac{c}{\max\{ c_0,1\}\,r^n}  \int_{B_r} \int_{B_r} \min_{y \in B_r}\left\{ \frac{c_0}{r^p} , V(y)\right\} \left| w(x)\right|^p \,dydx
\\
&\leq& c\left(\int_{B_r} \left|  Dw \right|^p \,dx + \int_{B_r}V \left|  w \right|^p  \,dx\right),
\end{eqnarray*}
which implies that
\begin{equation}\label{lem41pf1}
\frac{c_0}{\max\{ c_0,1\}} \mint_{B_r} \left|\frac{w}{r} \right|^{p} dx \leq c \left( \mint_{B_r} \left|  Dw \right|^p \,dx + \mint_{B_r}  V\left|  w \right|^p \,dx\right) .
\end{equation}

At this stage, if $c_0 < 1,$ we see that
\begin{eqnarray*}
\nonumber \left( \mint_{B_r} r^p\,V \, dx \right)^{\frac1p} \left( \mint_{B_r} \left| \frac{w}{r}\right|^{p} dx\right)^{\frac1p} &= & \left( c_0\mint_{B_r} \left| \frac{w}{r}\right|^{p} dx\right)^{\frac1p} \\
&& \hspace{-2cm}\leq\ c\left[ \left( \mint_{B_r} \left|  Dw \right|^p \,dx\right)^{\frac1p} + \left(\mint_{B_r} V \left|  w \right|^p  \,dx\right)^{\frac1p} \right].
\end{eqnarray*}
Using this and the fact $V\in \mathcal{B}_\gamma$, we have
\begin{eqnarray}\label{pc0>1}
\nonumber\left( \mint_{B_{r}}\left| \frac{w}{r}\right|^{p} \,dx\right)^{\frac1p} \left( \mint_{B_{r}} [r^pV]^{\gamma}\,dx\right)^{\frac{1}{p\gamma }} &&\\
\nonumber&&\hspace{-6.3cm} \leq b_{\gamma}^{\frac1p} \left( \mint_{B_{r}}\left| \frac{w}{r}\right|^{p} dx \right)^{\frac1p}  \left( \mint_{B_{r}}r^p V \, dx \right)^{\frac1p}\\
&&\hspace{-6.3cm} \leq c b_{\gamma}^{\frac1p}  \left[ \left( \mint_{B_{r}} \left|  Dw \right|^p \,dx\right)^{\frac1p} +   \left(\mint_{B_{r}} V \left|  w \right|^p  \,dx\right)^{\frac1p} \right].
 \end{eqnarray}
Otherwise, that is, if $c_0 \geq 1,$ we see from \eqref{lem41pf1} that
\begin{equation}\label{pc0<1}
\left( \mint_{B_r} \left| \frac{w}{r}\right|^{p} dx\right)^{\frac1p} \leq c \left[ \left( \mint_{B_r} \left|  Dw \right|^p \,dx\right)^{\frac1p} + \left(\mint_{B_r} V \left|  w \right|^p\,dx\right)^{\frac1p}\right].
\end{equation}

Then combining \eqref{pc0>1} and \eqref{pc0<1}, we finally obtain the desired estimate \eqref{prVbddpq}.
 \end{proof}

Now, let us consider a weak solution $w\in W^{1,p}(\Omega_{4r}(x_0))$ of
\begin{equation}
\label{homoeq1}
\left\{\begin{array}{rclcc}
-\mathrm{div}\, \mathbf{a}(x,Dw) + V |w|^{p-2}w&  =  &0 & \textrm{ in } & \Omega_{4r}(x_0),  \\
w & = & 0 & \textrm{ on } &  \partial_w\Omega_{4r}(x_0)\ \text{if}\ B_{4r}(x_0)\not\subset\Omega,\end{array}\right.
\end{equation}
and then we can obtain its gradient estimates as follows.

\begin{lemma}\label{lem42}
Let $1 < p < \infty$, and suppose that $\ba:\mr^n\times\mr^n\to\mr^n$ satisfies \eqref{aas1} and \eqref{aas2} and $V \in \mathcal{B}_{\gamma}$ for some $ \gamma \in [ \frac{n}{p}, n)$ when $p<n$ and for some $\gamma \in (1,n)$ when $p\geq n.$ There exists a small $\delta = \delta(n, p, \gamma, \Lambda, \nu) > 0 $ so that if $\mathbf{a}$ is $( \delta, R)$-vanishing  and $\Omega$ is $(\delta,R)$-Reifenberg flat for some $R\in(0,1)$, then for any $x_0\in\overline\Omega$, $r\in(0,\frac{R}{4}]$ satisfying $ (4r)^{p - \frac{n}{\gamma}} \Vert V \Vert_{L^{\gamma}(\Omega_{4r}(x_0))} \leq 1,$ and for any weak solution $w\in W^{1,p}(\Omega_{4r}(x_0))$ of \eqref{homoeq1} we have $|Dw|\in L^{\gamma^*(p-1)}(\Omega_r(x_0))$ with the estimate
\begin{eqnarray}\label{DwDwVwest}
\nonumber&&\left( \mint_{{\Omega}_r(x_0)} \left| D{w}\right|^{\gamma^* (p-1)}  \, dx \right)^{\frac{1}{\gamma^* (p-1)}}\\
&&\qquad  \leq c \left( \mint_{\Omega_{4r}(x_0)} \left|  Dw \right|^p \,dx\right)^{\frac1p} +   c \left(\mint_{\Omega_{4r}(x_0)} V\left|  w \right|^p  \,dx\right)^{\frac1p}.
\end{eqnarray}
Moreover, we have $V^{\frac1p}|w|\in L^{p\gamma }(\Omega_r(x_0))$ with the estimate
\begin{eqnarray}\label{DwDwVwest1}
\nonumber&&\left( \mint_{{\Omega}_r(x_0)} \left[ V^\frac{1}{p}|w|\right]^{p\gamma }  \, dx \right)^{\frac{1}{p\gamma }}\\
&&\qquad  \leq c\left( \mint_{\Omega_{4r}(x_0)} \left|  Dw \right|^p \,dx\right)^{\frac1p} + c  \left(\mint_{\Omega_{4r}(x_0)} V\left|  w \right|^p  \,dx\right)^{\frac1p}.
\end{eqnarray}
Here, the constants $c>0$ in the above estimates depend on $n,p,\gamma,\nu,L$ and $b_\gamma$.
\end{lemma}

\begin{proof}
For simplicity we shall denote $\Omega_{\rho} :=\Omega_{\rho}(x_0)$ and $B_{\rho} := B_{\rho}(x_0)$ for any $\rho>0$ in this proof. We first observe that, in view of Lemma \ref{supvlem} with $r$ replaced by $2r$, \begin{equation}\label{lem42pf1}
\|w\|_{L^\infty(\Omega_{2r})} \leq c \left( \mint_{\Omega_{4r}} |w|^p \, dx \right)^{\frac{1}{p}}.
\end{equation}
Then from the fact $V\in L^\gamma(\Omega),$ we see that $V|w|^{p-2}w\in L^\gamma(\Omega_{2r})$. Therefore, applying Theorem \ref{thmDwbdd} with $q=\gamma^*(p-1)$, $f=V|w|^{p-2}w$ and $F=0$, we have
\begin{eqnarray}\label{DwDwrVwes}
 \nonumber&& \left(\mint_{\Omega_r} |Dw|^{\gamma^*(p-1)} \, dx\right)^{\frac{1}{\gamma^*(p-1)}} \\  
 &&\qquad \leq c \left( \mint_{\Omega_{2r}} |Dw|^{p}\, dx\right)^{\frac{1}{p}}+ c \left(   \mint_{\Omega_{2r}} \left[r V|w|^{p-1}  \right]^\gamma \, dx\right)^{\frac{1}{\gamma(p-1)}}.
\end{eqnarray}
We now estimate the last term on the right hand side in the previous inequality. Using \eqref{lem42pf1} and \eqref{prVbddpq} with the assumption $ (4r)^{p - \frac{n}{\gamma}} \Vert V \Vert_{L^{\gamma}(\Omega_{4r})} \leq 1,$ we have
\begin{eqnarray}\label{rVwDwVwes}
\nonumber  \left( \mint_{\Omega_{2r}} \left[ r V |w|^{p-1}\right]^{\gamma} dx\right)^{\frac{1}{\gamma(p-1)}}
\nonumber&  \leq& \frac{ \|w\|_{L^\infty(\Omega_{2r})}}{r}  \leq c \left( \mint_{\Omega_{4r}}\left| \frac{w}{r}\right|^{p} \,dx\right)^{\frac1p}  \\
 &  \leq&  c \left[ \left( \mint_{\Omega_{4r}} \left|  Dw \right|^p \,dx\right)^{\frac1p} +   \left(\mint_{\Omega_{4r}} V\left|  w \right|^p  \,dx\right)^{\frac1p} \right].
  \end{eqnarray}
Here, we let $w\equiv 0$ in $B_{4r}\setminus \Omega$ and have used \eqref{dencon}.  Hence, inserting \eqref{rVwDwVwes} into \eqref{DwDwrVwes}, we obtain \eqref{DwDwVwest}. In the same way as \eqref{rVwDwVwes}, we can derive \eqref{DwDwVwest1}.
\end{proof}

\section{Comparison estimates}
\label{secComestimates}

In this section, we shall derive comparison estimates between the weak solution to \eqref{maineq} and weak solutions to localized equations  of \eqref{maineq} with $F\equiv0.$

\begin{lemma}\label{comestlem}
Let $1 < p < \infty$, and  suppose that $\mathbf{a}:\mr^n\times\mr^n\to\mr^n$ satisfies \eqref{aas1}-\eqref{aas2}.
 For any $\epsilon \in (0,1),$ there exists a small $ \delta = \delta( \epsilon, n, p, L, \nu) \in(0,1) $ such that if $u \in W^{1,p}(\Omega)$ is the weak solution to \eqref{maineq} with
\begin{equation}\label{lDuurbdd}
\left( \mint_{\Omega_{4r}} \left[|Du| + V^\frac{1}{p}|u|\right]^p \,dx\right)^{\frac{1}{p}} <\lambda
\end{equation}
and
\begin{equation}\label{lFbdd}
\left( \mint_{\Omega_{4r}} |F|^p \,dx\right)^{\frac{1}{p}}<\delta\, \lambda
\end{equation}
for some $r>0$ and $\lambda>0$, then  we have
\begin{equation}\label{lDumDvibdd}
\mint_{\Omega_{4r}}  |Du-Dw|^{p} + V |u-w|^{p} \, dx \leq \epsilon \lambda^p,
\end{equation}
where  $w \in W^{1,p}(\Omega_{4r})$ is the unique weak solution to
\begin{equation}
\label{lhomoeq}
\left\{\begin{array}{rclcc}
-\mathrm{div}\, \mathbf{a}(x,Dw) + V |w|^{p-2}w&  =  & 0 & \textrm{ in } & \Omega_{4r},  \\
w & = & u & \textrm{ on } & \partial \Omega_{4r}.
\end{array}\right.
\end{equation}
\end{lemma}

\begin{proof}
We first test the equations \eqref{lhomoeq} with the testing function $\varphi=w-u$ in order to discover
$$
\int_{\Omega_{4r}} \mathbf{a}(x, Dw)\cdot ( Dw-Du)\, dx + \int_{\Omega_{4r}} V |w|^{p-2}w  \cdot (w-u)\, dx= 0,
$$
and then, in view of \eqref{aas1} and \eqref{mono}, we obtain
$$
\int_{\Omega_{4r}} |Dw|^p\, dx + \int_{\Omega_{4r}} V |w|^p\, dx\leq c\int_{\Omega_{4r}} |Dw|^{p-1}|Du|\, dx + c \int_{\Omega_{4r}} V |w|^{p-1}|u|\, dx.
$$
Therefore, applying Young's inequality and \eqref{lDuurbdd} we have
\begin{equation}\label{lem51pf1}
\left(\int_{\Omega_{4r}} |Dw|^p\, dx \right)^{\frac1p}+\left( \int_{\Omega_{4r}} V |w|^p\, dx\right)^{\frac1p}\leq c\lambda.
\end{equation}

We next test the equations \eqref{maineq} and \eqref{lhomoeq} with the testing function $\varphi=u-w$ in order to discover
\begin{eqnarray*}%\label{weakformu-v}
\nonumber &&\int_{\Omega_{4r}} \left( \mathbf{a}(x, Du) - \mathbf{a}(x, Dw) \right) \cdot ( Du-Dw)\, dx \\
&& \quad+ \int_{\Omega_{4r}} V \left( |u|^{p-2}u - |w|^{p-2}w \right) \cdot (u-w)\, dx= \int_{\Omega_{4r}} |F|^{p-2}F \cdot (Du-Dw)\, dx.
\end{eqnarray*}
By virtue of the monotonicity condition \eqref{mono}, we derive that
\begin{eqnarray*}
&&\mint_{\Omega_{4r}}\left( |Du|^{2}+ |Dw|^2 \right)^{\frac{p-2}{2}}  | Du-Dw|^2\, dx \\
&&\qquad +\mint_{\Omega_{4r}}V \left( |u|^{2}+ |w|^2 \right)^{\frac{p-2}{2}}  | u-w|^2\, dx \leq c \mint_{\Omega_{4r}} |F|^{p-1}|Du-Dw|\, dx.
\end{eqnarray*}
Note that if $p\geq 2$, by \eqref{mono1}  we have
$$
\mint_{\Omega_{4r}}|Du-Dw|^p\, dx +\mint_{\Omega_{4r}}V |u-w|^p\, dx \leq c \mint_{\Omega_{4r}} |F|^{p-1}|Du-Dw|\, dx.
$$
On the other hand, if $1<p<2$, then by Young's inequality we have
\begin{eqnarray*}
 |Du-Dw|^p &=&   |Du-Dw|^p \left( |Du|^2 + |Dw|^2 \right)^{\frac{p(p-2)}{4}} \left( |Du|^2 + |Dw|^2 \right)^{\frac{p(2-p)}{4}}\\
 &\leq& \kappa \left( |Du|^2 + |Dw|^2 \right)^{\frac{p}{2}} +c(\kappa) \left( |Du|^2 + |Dw|^2 \right)^{\frac{p-2}{2}}  |Du-Dw|^2
 \end{eqnarray*}
and
$$
 V |u-w|^p \leq \kappa V \left( |u|^2 + |w|^2 \right)^{\frac{p}{2}} +c(\kappa) V \left( |u|^2 + |w|^2 \right)^{\frac{p-2}{2}}  |u-w|^2
$$
for any small $\kappa>0.$  Therefore, combining the above results with \eqref{lDuurbdd}, we have that
\begin{eqnarray*}
\nonumber&&\mint_{\Omega_{4r}} |Du-Dw|^{p}  \, dx+\mint_{\Omega_{4r}} V|u-w|^{p}  \, dx \\
\nonumber&&\quad \leq  \kappa \mint_{\Omega_{4r}} (|Du|^2+|Dw|^2)^\frac{p}{2}   \, dx  +\kappa \mint_{\Omega_{4r}} V (|u|^2+|w|^2)^\frac{p}{2}   \, dx\\\
\nonumber&&\quad \quad+ c(\kappa) \mint_{\Omega_{4r}} \left( |Du|^2 + |Dw|^2 \right)^{\frac{p-2}{2}} |Du-Dw|^{2}    \, dx\\
\nonumber&&\quad \quad+  c(\kappa)  \mint_{\Omega_{4r}} V\left( |u|^2 + |w|^2 \right)^{\frac{p-2}{2}} |u-w|^{2}    \, dx \\
\nonumber&&\quad \leq c_1 \kappa \, \lambda^p+ c(\kappa) \mint_{\Omega_{4 r}} |F|^{p-1}|Du-Dw|\, dx
\end{eqnarray*}
for some $c_1=c_1(n,p,L,\nu)>0$ and $c(\kappa)= c(\kappa,n,p,L,\nu)\geq1$.

Therefore, in any case, we obtain
\begin{eqnarray*}
\nonumber&&\mint_{\Omega_{4r}} |Du-Dw|^{p}  \, dx+\mint_{\Omega_{4r}} V \left|u-w\right|^{p}  \, dx\\
&& \qquad\leq c_1\kappa \, \lambda^p+ c(\kappa,\tau)\mint_{\Omega_{4r}} |F|^{p}\, dx + c(\kappa)\,\tau \mint_{\Omega_{4r}} |Du-Dw|^p\, dx \\
& &  \qquad  \leq  c_1\kappa \, \lambda^p+ c(\kappa,\tau)\delta^p \,\lambda^p + c(\kappa)\,\tau \mint_{\Omega_{4r}}|Du-Dw|^p\, dx
\end{eqnarray*}
for any small $\kappa,\tau>0$ and for some $c(\kappa,\tau)=c(\kappa,\tau,n,p,L,\nu)\geq 1$. Here, we have used Young's inequality and \eqref{lFbdd}. Taking  $\kappa$, $\tau$ and $\delta$ sufficiently small such that
$$
\kappa=\frac{\epsilon}{4c_1},\ \ \ \tau = \frac{1}{2c(\kappa)}\ \ \ \text{and}\ \ \ \delta =\left(\frac{\epsilon}{4c(\kappa,\tau)}\right)^{\frac1p},
$$
we finally obtain \eqref{lDumDvibdd}.
\end{proof}

 We notice that $\gamma^*(p-1)>\max\{p,\frac{n(p-1)}{n-1}\}$. Therefore, applying the results in Lemma \ref{lem42} to the weak solution $w$ of \eqref{lhomoeq} in the previous lemma, we obtain the following gradient estimates.

\begin{lemma}\label{comestlem1}
Let $1 < p < \infty$, and  suppose that $\mathbf{a}:\mr^n\times\mr^n\to\mr^n$ satisfies \eqref{aas1}-\eqref{aas2} and $V \in \mathcal{B}_\gamma$ for some $\gamma \in [ \frac{n}{p}, n)$ when $p<n$ and for some $\gamma \in (1,n)$ when $p\geq n.$
There exists a small $ \delta = \delta(n, p, \gamma, L, \nu) > 0 $ such that if $\mathbf{a}$ is $( \delta, R)$-vanishing, $\Omega$ is $( \delta, R)$-Reifenberg flat and  $u \in W^{1,p}(\Omega)$ is the weak solution to \eqref{maineq} with \eqref{lDuurbdd} and \eqref{lFbdd} for some $r\in(0,\frac R4]$ satisfying $ (4r)^{p - \frac{n}{\gamma}} \Vert V \Vert_{L^{\gamma}(\Omega_{4r})} \leq 1$ and $\lambda>0$, then  we have
$$
\left( \mint_{\Omega_{r}} \left| Dw\right|^{\gamma^* (p-1)}  \, dx \right)^{\frac{1}{\gamma^* (p-1)}} \leq c\, \lambda
$$
and
$$
 \left( \mint_{\Omega_{r}} \left[V^{\frac1p} |w|\right]^{p\gamma }  \, dx \right)^{\frac{1}{p\gamma }} \leq c\, \lambda
$$
for some $c=c(n,p, \gamma, L, \nu, b_\gamma)>0,$ where $w$ is the unique weak solution to \eqref{lhomoeq}.
\end{lemma}
\begin{proof}
The estimates above directly follow from  \eqref{DwDwVwest}, \eqref{DwDwVwest1} and \eqref{lem51pf1}.
\end{proof}

\section{$L^q$-estimates}
\label{sec gradient estimates}
Now we are ready to prove  our main results,  Theorem~\ref{mainthm} and Corollary~\ref{maincor}. As we mentioned in the introduction, we employ so-called an exit-time argument introduced by Mingione in \cite{AM1,Min1}.  
\subsection{Proof of Theorem \ref{mainthm}}\ \\
Assume that $\ba:\mr^n\times \mr^n\to\mr^n$ is $(\delta,R)$-vanishing and $\Omega$ is $(\delta,R)$-Reifenberg flat for some $R\in(0,1),$ where   $\delta\in(0,1)$  will be chosen sufficiently small later.
 Now, we prove the estimate \eqref{mainlocalest}. Fix any $x_0\in\overline{\Omega}$ and $r>0$ satisfying $r \leq \frac{R}{4}$  and $(4r)^{p - \frac{n}{\gamma}} \Vert V \Vert_{L^{\gamma}(\Omega_{4r}(x_0))} \leq 1.$ Note that 
 $$
 \rho^{p - \frac{n}{\gamma}} \Vert V \Vert_{L^{\gamma}(\Omega_{\rho}(y))}\leq 1
 $$ 
 for any $B_\rho(y)\subset B_{4r}(x_0)$ with $y\in B_{4r}(x_0)\cap \overline{\Omega}.$
 For the sake of simplicity, we shall write $\Omega_{\rho}:=\Omega_{\rho}(x_0)$, $\rho>0$. Also, we define
\begin{equation}\label{Phi}
\Phi(u,V):= |Du|+V^\frac{1}{p}|u|,
\end{equation}
and for $\lambda,\rho>0$
$$
E(\lambda, \rho) := \{x \in \Omega_{\rho} :    \Phi(u,V)(x)>\lambda  \}.
$$

The proof goes in five steps.

\vspace{0.2cm}\textit{Step 1. Covering argument.}\ \\
Fix any $s_1, s_2$ with $1\leq s_1 < s_2 \leq 2.$ Then we have $ \Omega_{r}\subset \Omega_{s_1 r } \subset  \Omega_{s_2 r } \subset  \Omega_{2 r } $. We define
 \begin{equation}\label{lambda0}
\lambda_0:= \left( \mint_{\Omega_{2r}} \left[\Phi(u,V)\right]^p \,dx\right)^{\frac1p}+\frac{1}{\delta} \left(
 \mint_{\Omega_{2r}} |F|^p \,dx \right)^{\frac1p},
\end{equation}
and consider $\lambda>0$ large enough so that
\begin{equation}\label{lambdarg}
\lambda > A\, \lambda_0,\ \ \textrm{where } A : = \left( \frac{16}{7}\right)^{\frac{n}{p}}\left( \frac{40}{s_2 - s_1}\right)^{\frac{n}{p}}.
\end{equation}
We note that
$ \Omega_{\rho}(y) \subset \Omega_{2r}$ for any $y \in E(\lambda, s_1 r)$ and any $ \rho \in \left( 0, (s_2-s_1)\,r \right].$
By virtue of the measure density condition \eqref{dencon} and  the definition of $\lambda_0$ in \eqref{lambda0}, we then deduce that
\begin{eqnarray*}
&&\left(\mint_{\Omega_{\rho}(y)} \left[\Phi(u,V)\right]^p \,dx\right)^\frac1p+
\frac{1}{\delta}\left( \mint_{\Omega_{\rho}(y)} |F|^p \,dx\right)^\frac1p\\
&& \leq \left(\frac{|\Omega_{2r}|}{|\Omega_{\rho}(y)|}\right)^{\frac{1}{p}} \lambda_0  \leq \left( \frac{16}{7}\right)^{\frac np} \left( \frac{2r}{\rho}\right)^{\frac np}\, \lambda_0 \leq A\, \lambda_0  < \lambda,
\end{eqnarray*}
provided that
$$ \frac{(s_2 - s_1)\,r}{20} \leq \rho \leq (s_2-s_1)\,r.$$

On the other hand, Lebesgue's differentiation theorem yields that for almost every $y \in E(\lambda, s_1 r),$
$$ \lim_{\rho \rightarrow 0} \left[\left(\mint_{\Omega_{\rho}(y)} \left[\Phi(u,V)\right]^p \,dx \right)^{\frac1p}+\frac{1}{\delta} \left(
 \mint_{\Omega_{\rho}(y)} |F|^p \,dx\right)^{\frac1p}\right] > \lambda.$$
Therefore the continuity of the integral implies that for almost every $y \in E(\lambda, s_1 r),$ there exists
$$ \rho_{y} = \rho(y) \in \left( 0, \frac{(s_2-s_1)\,r}{20}\right)$$
such that
$$ \left(\mint_{\Omega_{\rho_y}(y)} \left[\Phi(u,V)\right]^p \,dx \right)^{\frac1p}+\frac{1}{\delta} \left(
 \mint_{\Omega_{\rho_y}(y)} |F|^p \,dx\right)^{\frac1p} = \lambda$$
and, for any $ \rho \in ( \rho_y, (s_2-s_1)r],$
$$ \left(\mint_{\Omega_{\rho}(y)} \left[\Phi(u,V)\right]^p \,dx \right)^{\frac1p}+\frac{1}{\delta} \left(
 \mint_{\Omega_{\rho}(y)} |F|^p \,dx\right)^{\frac1p} <\lambda.
 $$

Applying Vitali's covering theorem, we have the following:
\begin{lemma}\label{coveringlem}
Given $\lambda > A\, \lambda_0,$ there exists a disjoint family of $\{ \Omega_{\rho_i}(y^i)\}_{i=1}^{\infty}$ with $y^i \in E(\lambda, s_1 r)$ and $\rho_{i} \in \left(0, \frac{(s_2-s_1)\,r}{20} \right)$ such that
$$E(\lambda, s_1 r) \subset \bigcup_{i=1}^{\infty} \Omega_{5\rho_i}(y^i), $$
\begin{equation}\label{covering1}
 \left(\mint_{\Omega_{\rho_i}(y^i)} \left[\Phi(u,V)\right]^p \,dx\right)^{\frac1p} +
\frac{1}{\delta}\left( \mint_{\Omega_{\rho_i}(y^i)} |F|^p \,dx\right)^\frac1p = \lambda,
\end{equation}
and  for any $ \rho \in ( \rho_i, (s_2-s_1)\,r]$,
\begin{equation}\label{covering2}
\left(\mint_{\Omega_{\rho_i}(y^i)} \left[\Phi(u,V)\right]^p \,dx\right)^{\frac1p} + \frac{1}{\delta}\left( \mint_{\Omega_{\rho_i}(y^i)} |F|^p \,dx\right)^\frac1p<\lambda.
\end{equation}
\end{lemma}

Furthermore, we can deduce  from Lemma~\ref{coveringlem}, in particular, \eqref{covering1}, that
\begin{equation}\label{covering3}
\mint_{\Omega_{\rho_i}(y^i)} \left[\Phi(u,V)\right]^p \,dx \geq \left(\frac{\lambda}{2}\right)^p\ \ \ \text{or}\ \ \
 \mint_{\Omega_{\rho_i}(y^i)} |F|^p \,dx \geq \left(\frac{\delta\lambda}{2}\right)^p.
\end{equation}
If the first inequality holds, we have
$$
 \left|\Omega_{\rho_i}(y^i)\right|\leq  \frac{2^p}{\lambda^p}\left( \int_{\Omega_{\rho_i}(y^i)\cap\left\{\Phi(u,V)>\frac{\lambda}{2^{(p+1)/p}} \right\}} \left[\Phi(u,V)\right]^p \,dx+\frac{\left|\Omega_{\rho_i}(y^i)\right|\lambda^p}{2^{p+1}} \right)
$$
and so
$$
 \left|\Omega_{\rho_i}(y^i)\right|\leq  \frac{2^{p+1}}{\lambda^p} \int_{\Omega_{\rho_i}(y^i)\cap\left\{\Phi(u,V)>\frac{\lambda}{2^{(p+1)/p}} \right\}} \left[\Phi(u,V)\right]^p \,dx.
$$
Similarly, if the second inequality in \eqref{covering3} holds, we have
$$
 \left|\Omega_{\rho_i}(y^i)\right|\leq  \frac{2^{p+1}}{(\delta\lambda)^p} \int_{\Omega_{\rho_i}(y^i)\cap\left\{|F|>\frac{\lambda\delta}{2^{(p+1)/p}} \right\}} |F|^p \,dx.
$$
Therefore, in any case, we have
\begin{eqnarray}\label{omegai}
\nonumber \left|\Omega_{\rho_i}(y^i)\right|&\leq & \frac{2^{p+1}}{\lambda^p} \Bigg(\int_{\Omega_{\rho_i}(y^i)\cap\left\{\Phi(u,V)>\frac{\lambda}{2^{(p+1)/p}} \right\}} \left[\Phi(u,V)\right]^p \,dx\\
 &&\qquad \qquad \quad+ \int_{\Omega_{\rho_i}(y^i)\cap\left\{\frac{|F|}{\delta}>\frac{\lambda}{2^{(p+1)/p}} \right\}} \left[\frac{|F|}{\delta}\right]^p \,dx\Bigg).
\end{eqnarray}

\vspace{0.2cm}\textit{Step 2. Comparison estimates.}\ \\
From Lemma~\ref{coveringlem}, in particular, \eqref{covering2}, we note that
$$
\left(\mint_{\Omega_{20\rho_i}(y^i)} \left[|Du|+  V^{\frac1p} |u|\right]^p \,dx\right)^\frac1p <\lambda \ \ \ \text{and}\ \ \
 \left(\mint_{\Omega_{20\rho_i}(y^i)} |F|^p \,dx\right)^\frac1p <\delta\lambda.
$$
Then applying Lemma~\ref{comestlem} and Lemma~\ref{comestlem1}, for any $\epsilon \in (0,1),$  there exists a small $\delta= \delta( \epsilon, n, p, \gamma,  L, \nu)>0$ such that
\begin{equation}\label{tDumDvibdd}
\left( \mint_{\Omega_{20\rho_i}(y^i)} |Du-Dw_i|^{p}  \, dx\right)^{\frac1p} +
\left( \mint_{\Omega_{20\rho_i}(y^i)} \left|V^{\frac1p}u-V^{\frac1p}w_i\right|^{p}  \, dx\right)^{\frac1p} \leq \epsilon \lambda,
\end{equation}
\begin{equation}\label{tDvihibdd}
\left( \mint_{\Omega_{5\rho_i}(y^i)} \left| Dw_i\right|^{\gamma^* (p-1)}  \, dx \right)^{\frac{1}{\gamma^* (p-1)}} \leq c\lambda
\end{equation}
and
\begin{equation}\label{tDvihibdd1}
\left( \mint_{\Omega_{5\rho_i}(y^i)} \left[V^\frac{1}{p}\left| w_i\right|\right]^{p\gamma }  \, dx \right)^{\frac{1}{p\gamma }} \leq c\lambda,
\end{equation}
where $w_i \in W^{1,p}(\Omega_{20\rho_i}(y^i))$ is the unique weak solution to
$$
\left\{\begin{array}{rclcc}
-\mathrm{div}\, \mathbf{a}(x,Dw_i) + V |w_i|^{p-2}w_i&  =  & 0 & \textrm{ in } & \Omega_{20\rho_i}(y^i),  \\
w_i & = & u & \textrm{ on } & \partial \Omega_{20\rho_i}(y^i).
\end{array}\right.
$$
Furthermore, recalling the definition of $\Phi$ in \eqref{Phi} and the fact $\gamma^* (p-1)>p\gamma $, we have from \eqref{tDvihibdd} and  \eqref{tDvihibdd1} that
\begin{equation}\label{Vvibdd}
\mint_{\Omega_{5\rho_i}(y^i)} \left[\Phi(w_i,V)\right]^{p\gamma }  \, dx \leq c\lambda^{p\gamma},
\end{equation}
for some constant  $c=c(n,p, \gamma ,L, \nu,b_\gamma)>0.$

\vspace{0.2cm}\textit{Step 3. Estimates for $\Phi(u,V)$.} \ \\
Let $y \in \Omega_{5\rho_i}(y^i)$ such that $\Phi(u,V)(y)> K\lambda,$ where $K\geq 1$ will be chosen later.
We then note that
$$
\Phi(u,V)(y) \leq \Phi(w_i,V)(y)+|Du(y)-Dw_i(y)|+[V(y)]^{\frac1p}|u(y)-w_i(y)|.
$$
Here, we need to consider the two cases:
\begin{eqnarray*}
\textrm{(i)} && \Phi(w_i,V)(y) \leq|Du(y)-Dw_i(y)|+[V(y)]^{\frac1p}|u(y)-w_i(y)|,\\
\textrm{(ii)} && \Phi(w_i,V)(y)  > |Du(y)-Dw_i(y)|+[V(y)]^{\frac1p}|u(y)-w_i(y)|.
\end{eqnarray*}
For the case (i), it is clear that
$$
\Phi(u,V)(y) \leq 2\left(|Du(y)-Dw_i(y)|+[V(y)]^{\frac1p}|u(y)-w_i(y)|\right).
$$
For the case (ii), we have that
$$ K\lambda  < \Phi(u,V)(y)  \leq 2\Phi(w_i,V)(y),$$
from which, it follows that
$$
\Phi(u,V)(y)\leq 2 \Phi(w_i,V)(y) \left[ \frac{2\Phi(w_i,V)(y)}{K\lambda}   \right]^{\gamma -1}  = \frac{2^\gamma}{(K \lambda)^{\gamma-1}} [\Phi(w_i,V)(y)]^{\gamma} .
$$
 In turn, for the both cases (i) and (ii), we have that
\begin{eqnarray*}
[\Phi(u,V)(y)]^p &\leq&  2^{2p-1} \left( |Du(y)-Dw_i(y)|^p + V(y)|u(y)-w_i(y)|^p \right)\\
 && + \frac{2^{p\gamma}}{(K \lambda)^{p\gamma-p}} [\Phi(w_i,V)(y)]^{p\gamma}
\end{eqnarray*}
for any $y \in \Omega_{5\rho_i}(y^i)$ such that $\Phi(u,V)(y)> K\lambda.$

Then applying \eqref{tDumDvibdd}-\eqref{Vvibdd}, we deduce that
\begin{eqnarray*}
 \int_{\Omega_{5\rho_i}(y^i) \cap E(K\lambda , s_2 r)} [\Phi(u,V)]^p\, dx &\leq &  c \int_{\Omega_{5\rho_i}(y^i)} \left[  |Du - Dw_i|^p +V |u- w_i|^p \right]\, dx\\
&&\quad +  \frac{c}{(K \lambda)^{p\gamma-p}}  \int_{\Omega_{5\rho_i}(y^i)} [\Phi(w_i,V)(y)]^{p\gamma}  \, dx \\
& \leq &  c \left(\epsilon \lambda^p  +  \frac{\lambda^{p\gamma}  }{(K \lambda)^{p\gamma-p}}\right) \left| \Omega_{5\rho_i}(y^i)\right|\\
& \leq &  c\, \lambda^p \left(  \epsilon + \frac{1}{K^{p\gamma-p}}  \right) \left| \Omega_{\rho_i}(y^i)\right|\\
& =  & c\,\tilde\epsilon \lambda^p \left| \Omega_{\rho_i}(y^i)\right|
\end{eqnarray*}
for some constant $c=c(n,p, \gamma, L, \nu, b_\gamma)>0,$ where
\begin{equation}\label{tepsilon}
\tilde \epsilon:=   \epsilon + \frac{1}{K^{p\gamma-p}}.
\end{equation}
Therefore, inserting \eqref{omegai} into the previous estimate, we have that
\begin{eqnarray*}
 \int_{\Omega_{5\rho_i}(y^i) \cap E(K\lambda, s_1 r)} [\Phi(u,V)]^p\, dx & \leq& c\tilde \epsilon  \Bigg(  \int_{\Omega_{\rho_i}(y^i)\cap\left\{\Phi(u,V)>\frac{\lambda}{2^{(p+1)/p}} \right\}} \left[\Phi(u,V)\right]^p \,dx\\
&&\qquad\quad + \int_{\Omega_{\rho_i}(y^i)\cap\left\{\frac{|F|}{\delta}>\frac{\lambda\delta}{2^{(p+1)/p}} \right\}} \left[\frac{|F|}{\delta}\right]^p \,dx\Bigg).
\end{eqnarray*}

According to Lemma~\ref{coveringlem}, we note that $\Omega_{\rho_i}(y^i)$ is mutually disjoint and
$$
E(K\lambda, s_1 r) \subset E(\lambda, s_1 r )\subset \bigcup_{i=1}^{\infty} \Omega_{5\rho_i}(y^i) \subset \Omega_{s_2 r} ,
$$
since $K \geq 1.$ Then  we have that
\begin{eqnarray}\label{EDubddcal}
\nonumber \int_{ E(K\lambda, s_1 r)} [\Phi(u,V)]^p\, dx  &\leq&  \sum_{i=1}^{\infty}  \int_{\Omega_{5\rho_i}(y^i) \cap E(K\lambda, s_1 r)} [\Phi(u,V)]^p\, dx \\
\nonumber & \leq & c\tilde \epsilon  \Bigg(  \int_{\Omega_{s_2r}\cap\left\{\Phi(u,V)>\frac{\lambda}{2^{(p+1)/p}} \right\}} \left[\Phi(u,V)\right]^p \,dx\\
 & &\qquad \qquad + \int_{\Omega_{s_2r}\cap\left\{\frac{|F|}{\delta}>\frac{\lambda}{2^{(p+1)/p}} \right\}} \left[\frac{|F|}{\delta}\right]^p \,dx\Bigg)
\end{eqnarray}
for some constant  $c=c(n,p, \gamma, L, \nu, b_\gamma)>0.$

%%%%%%%%%%%%%%%%%%%

\vspace{0.2cm}\textit{Step 4. Proof of \eqref{mainlocalest} when $q\in(p,p\gamma)$.} \ \\
We shall use a truncation argument. For $k \geq A\lambda_0,$ let us define
$$
\Phi(u,V)_k := \min\left\{ \Phi(u,V), k \right\},
$$
and denote the upper level sets with respect to $\Phi(u,V)_k $ by
$$ E_k ( \tilde \lambda, \rho):= \left\{ y \in \Omega_{\rho} :  \Phi(u,V)_k> \tilde \lambda\right\}\ \ \text{for }\tilde \lambda,\rho>0. $$
Then since $E_k ( K\lambda, s_1 r) \subset E ( K\lambda, s_1 r)$ and
$$
\left\{ \Phi(u,V)_k >\frac{\lambda}{2^{(p+1)/p}} \right\} = \left\{ \Phi(u,V)>\frac{\lambda}{2^{(p+1)/p}} \right\},
$$
we see from \eqref{EDubddcal} that
\begin{eqnarray*}
 \int_{ E_k(K\lambda, s_1 r)} [\Phi(u,V)]^p\, dx & \leq & c\tilde \epsilon  \Bigg(  \int_{E_k\left(\frac{\lambda}{2^{(p+1)/p}},s_2r\right)} \left[\Phi(u,V)\right]^p \,dx\\
 & &\qquad \quad +\int_{\Omega_{s_2r}\cap\left\{\frac{|F|}{\delta}>\frac{\lambda}{2^{(p+1)/p}} \right\}} \left[\frac{|F|}{\delta}\right]^p \,dx\Bigg).
\end{eqnarray*}
Then by multiplying both sides by $\lambda^{q-p-1}$ and integrating with respect to $\lambda$ over $(A\lambda_0, \infty)$, we have that
\begin{eqnarray}\label{lamDuEk}
\nonumber I_0 &:=&\int_{A\lambda_0}^\infty \lambda^{q-p-1}\int_{ E_k(K\lambda, s_1 r)} [\Phi(u,V)]^p\, dxd\lambda\\
 \nonumber & \leq & c\tilde \epsilon  \Bigg(  \int_{A\lambda_0}^\infty \lambda^{q-p-1} \int_{E_k\left(\frac{\lambda}{2^{(p+1)/p}},s_2r\right)} \left[\Phi(u,V)\right]^p \,dxd\lambda\\
\nonumber & &\qquad \quad  +\int_{A\lambda_0}^\infty \lambda^{q-p-1}  \int_{\Omega_{s_2r}\cap\left\{\frac{|F|}{\delta}>\frac{\lambda}{2^{(p+1)/p}} \right\}} \left[\frac{|F|}{\delta}\right]^p \,dxd\lambda\Bigg)\\
&=:& c\tilde\epsilon (I_1+I_2).
\end{eqnarray}
Here, Fubini's theorem allows us to deduce that
\begin{eqnarray*}
I_0 &=& \int_{E_k(KA\lambda_0, s_1 r)} [\Phi(u,V)]^p  \left( \int_{A\lambda_0}^{\Phi(u,V)_k(x)/K}  \lambda^{q-p-1} \, d\lambda \right) dx\\
&=& \frac{1}{q-p}\, \Bigg\{  \int_{E_k(KA\lambda_0, s_1 r)} [\Phi(u,V)]^p  \left[\frac{\Phi(u,V)_k}{K}\right]^{q-p}  \,  dx \\
&&\qquad\qquad\qquad\qquad  -(A\lambda_0)^{q-p}\int_{E_k(KA\lambda_0, s_1 r)} [\Phi(u,V)]^p \,dx \Bigg\}.
\end{eqnarray*}
We also employ Fubini's theorem to discover
\begin{eqnarray*}
I_1&=&     \int_{E_k\left(\frac{A\lambda_0}{2^{(p+1)/p}},s_2r\right)} [\Phi(u,V)]^p \left(\int_{A\lambda_0}^{ 2^{(p+1)/p} \Phi(u,V)_k(x)}  \lambda^{q-p-1} \,d\lambda\right) dx \\
&\leq&  \frac{1}{q-p}  \int_{E_k\left(\frac{A\lambda_0}{2^{(p+1)/p}},s_2r\right)} [\Phi(u,V)]^p  \left[2^{(p+1)/p} \Phi(u,V)_k\right]^{q-p}\, dx \\
&\leq &  \frac{2^{(p+1)(q-p)/p}}{q-p}   \int_{\Omega_{s_2 r}}  [\Phi(u,V)]^p  \left[\Phi(u,V)_k\right]^{q-p} \, dx.
\end{eqnarray*}
Similarly, we obtain that
$$
I_2\leq  \frac{2^{(p+1)(q-p)/p}}{q-p}   \int_{\Omega_{s_2 r}}  \left[\frac{|F|}{\delta}\right]^q  \, dx.
$$
Therefore, inserting the previous estimates for $I_0$, $I_1$, $I_2$ into \eqref{lamDuEk}, we derive
\begin{eqnarray*}
&&\int_{E_k(KA\lambda_0, s_1 r)} [\Phi(u,V)]^p  [\Phi(u,V)_k]^{q-p}  \,  dx \\
  && \quad \leq  (KA\lambda_0)^{q-p}\int_{\Omega_{s_1 r}} [\Phi(u,V)]^p \,dx\\
&&\quad \quad +c\tilde \epsilon K^{q-p} \left( \int_{\Omega_{s_2 r}}  [\Phi(u,V)]^p [\Phi(u,V)_k]^{q-p} \, dx + \int_{\Omega_{s_2 r}}   \left[\frac{|F|}{\delta} \right]^{q} \, dx \right).
\end{eqnarray*}
We also notice that
$$
\int_{\Omega_{s_1 r} \setminus E_k(KA\lambda_0, s_1 r)} [\Phi(u,V)]^p[\Phi_k(u,V)_k]^{q-p}\,  dx  \leq  (KA\lambda_0)^{q-p} \int_{\Omega_{s_1 r}} [\Phi(u,V)]^p \, dx.
$$
Finally, from the last two estimates we have that
\begin{eqnarray*}
\int_{\Omega_{s_1 r}}[\Phi(u,V)]^p [\Phi(u,V)_k]^{q-p}  \,  dx & \leq & (KA\lambda_0)^{q-p}\int_{\Omega_{s_1 r}}[\Phi(u,V)]^p\,dx\\
&&\hspace{-4.5cm}+c_2\tilde \epsilon K^{q-p} \left( \int_{\Omega_{s_2 r}}  [\Phi(u,V)]^p [\Phi(u,V)_k]^{q-p} \, dx + \int_{\Omega_{s_2 r}}   \left[\frac{|F|}{\delta} \right]^{q} \, dx \right)
\end{eqnarray*}
for some $c_2=c_2(n,p,\gamma,L,\nu,b_\gamma,q)>0$. At this stage, we recall the definition of $\tilde \epsilon$ in \eqref{tepsilon}, and then take large $K>1$ and small $\epsilon\in(0,1)$ depending on  $n,p,\gamma,L,\nu,b_\gamma,q$ such that
$$
K\geq (4c_2)^{\frac{1}{p\gamma-q}} \ \ \ \text{and}\ \ \ \epsilon\leq\frac{1}{4c_2K^{q-p}},
$$
hence $\delta=\delta(n,p,\gamma,L,\nu,b_\gamma,q)\in(0,1)$ is finally determined.  Consequently, recalling the definition of $A$ in \eqref{lambdarg} we  have
\begin{eqnarray*}
\int_{\Omega_{s_1 r}}[\Phi(u,V)]^p [\Phi(u,V)_k]^{q-p}  \,  dx & \leq & \frac12 \int_{\Omega_{s_2 r}}  [\Phi(u,V)]^p [\Phi(u,V)_k]^{q-p} \, dx \\
&& \hspace{-2cm}+\frac{c\lambda_0^{q-p}}{(s_2-s_1)^{\frac np}}\int_{\Omega_{2r}}[\Phi(u,V)]^p\,dx+ c \int_{\Omega_{2 r}}   |F|^q \, dx .
\end{eqnarray*}
Then applying Lemma~\ref{teclem}, we derive that
$$
\int_{\Omega_{r}}[\Phi(u,V)]^p [\Phi(u,V)_k]^{q-p}  \,  dx  \leq c\lambda_0^{q-p}\int_{\Omega_{2r}}[\Phi(u,V)]^p\,dx+ c \int_{\Omega_{2 r}}   |F|^q \, dx
$$
for any $k>A\lambda_0$. Finally, by Lebesgue's monotone convergence theorem, the definition of $\lambda_0$ in  \eqref{lambda0}, H\"older's inequality and Young's inequality, we obtain that
\begin{eqnarray*}
\mint_{\Omega_{r}}[\Phi(u,V)]^q\,dx &= &\lim_{k\to\infty} \mint_{\Omega_{r}}[\Phi(u,V)]^p [\Phi(u,V)_k]^{q-p}  \,  dx\\
& \leq& c\lambda_0^{q-p}\mint_{\Omega_{2r}}[\Phi(u,V)]^p\,dx+ c \mint_{\Omega_{2 r}}   |F|^q \, dx\\
& \leq& c\left(\mint_{\Omega_{2r}}[\Phi(u,V)]^p\,dx\right)^{\frac{q}{p}}+ c \mint_{\Omega_{2 r}}   |F|^q \, dx,
\end{eqnarray*}
and so, recalling the definition of $\Phi(u,V)$ in \eqref{Phi},
\begin{eqnarray}\label{mainlocalest1}
\nonumber\left( \mint_{\Omega_{r}} |Du|^q +\left[V^{\frac1p}|u|\right]^{q}\,  dx\right)^{\frac1q} &\leq& c \left( \mint_{\Omega_{2r}} |Du|^p + \left[ V^{\frac1p} |u|\right]^p \,dx \right)^{\frac{1}{p}}\\
&&+ c  \left(\mint_{\Omega_{2r}}   \left|F \right|^{q} \, dx\right)^{\frac1q},
\end{eqnarray}
which derives the estimate \eqref{mainlocalest} for $q\in(p,p\gamma)$.

\vspace{0.2cm}\textit{Step 5. Proof of \eqref{mainlocalest} when $q\in[p\gamma,\gamma^*(p-1))$.}  \ \\
Finally, we prove the estimate \eqref{mainlocalest} for the remaining   range of $q$.  Note that we only consider the gradient of $u$ since $\chi_{\{q<p\gamma\}}=0$.

We first suppose that $q\in[p\gamma,\gamma^*(p-1))$ satisfies
\begin{equation}\label{case1}
\max\left\{p,\frac{n(p-1)}{n-1}\right\}<q.
\end{equation}
Note that if $p\leq n$ we have $\max\{p,\frac{n(p-1)}{n-1}\}=p$ and so the previous inequality is trivial. Then let us set $\tilde q\in (1,\gamma)$  such that
\begin{equation}\label{tq}
q=(p-1)\tilde q^*.
\end{equation}
Then we see from H\"older's inequality that
 \begin{eqnarray*}
 \left( \mint_{\Omega_{2r}} \left[ r V |u|^{p-1}\right]^{\tilde{q}}\, dx \right)^{\frac{1}{\tilde{q}}} & = &\left( \mint_{\Omega_{2r}} \left[ r V^{\frac1p} \right]^{\tilde{q}} \left[ V^{\frac1p} |u| \right]^{(p-1)\tilde{q}}\, dx \right)^{\frac{1}{\tilde q}} \\
 &\leq&  \left( \mint_{\Omega_{2r}} [r^pV]^{\tilde{q}}\, dx\right)^{\frac{1}{p\tilde{q}}} \left( \mint_{\Omega_{2r}} \left[V^{\frac{1}{p}}|u|\right]^{p\tilde{q}} \, dx  \right)^{\frac{p-1}{p\tilde{q}}} \\
& \leq  & c \left( r^{p\gamma-n} \int_{\Omega_{2r}} V^{\gamma}\, dx\right)^{\frac{1}{p\gamma}} \left( \mint_{\Omega_{2r}} \left[V^{\frac{1}{p}}|u|\right]^{p\tilde{q}} \, dx  \right)^{\frac{p-1}{p\tilde{q}}}\\
& \leq  & c \left( \mint_{\Omega_{2r}} \left[V^{\frac{1}{p}}|u| \right]^{p\tilde{q}} \, dx  \right)^{\frac{p-1}{p\tilde{q}}}.
\end{eqnarray*}
Here we have used the facts that $\tilde{q} \in (1,\gamma)$ and $(2r)^{p - \frac{n}{\gamma}} \Vert V \Vert_{L^{\gamma}(\Omega_{2r})}\leq 1.$ Therefore, applying the estimate \eqref{mainlocalest1} with $q$ and $r$ replaced by $p\tilde q$  and $2r$, respectively, we have that  $V |u|^{p-2}u\in L^{\tilde q}(\Omega_{2r})$ with the estimate
\begin{eqnarray}\label{mainpf1}
\nonumber \left( \mint_{\Omega_{2r}} \left[ r V |u|^{p-1}\right]^{\tilde{q}}\, dx \right)^{\frac{1}{\tilde{q}}} &\leq& c  \left( \mint_{\Omega_{4r}}\left[ |Du|  + V^\frac{1}{p}|u|\right]^p \,dx \right)^{\frac{p-1}{p}}\\
&&+c  \left(\mint_{\Omega_{4 r}}   |F|^{p\tilde{q}} \, dx\right)^{\frac{p-1}{p\tilde{q}}}.
\end{eqnarray}
Finally, by Theorem~\ref{thmDwbdd} with $f=V |u|^{p-2}u$, the previous estimate \eqref{mainpf1} and H\"older's inequality, we have that
\begin{eqnarray}\label{mainpf2}
\nonumber \left(\mint_{\Omega_r} |Du|^{q} \, dx \right)^{\frac{1}{q}} &\leq& c \left( \mint_{\Omega_{2r}} |Du|^{p}\, dx\right)^{\frac{1}{p}}   + c \left( \mint_{\Omega_{2r}} \left[ r V |u|^{p-1}\right]^{\tilde{q}} \, dx\right)^{\frac{1}{\tilde{q}(p-1)}} \\
\nonumber &&  + c \left( \mint_{\Omega_{2r}} |F|^{q} \, dx\right)^{\frac{1}{q}}\\
\nonumber &\leq& c \, \left( \mint_{\Omega_{4r}} \left[ |Du|+ V^{\frac1p} |u|\right]^p \,dx \right)^{\frac{1}{p}} + c  \left(\mint_{\Omega_{4 r}}   \left|F \right|^{p \tilde{q}} \, dx\right)^{\frac{1}{p \tilde{q}}} \\
\nonumber && + c  \left(\mint_{\Omega_{4 r}}   \left|F \right|^{q} \, dx\right)^{\frac{1}{q}}\\
&\leq& c \, \left( \mint_{\Omega_{4r}} \left[ |Du|+ V^{\frac1p} |u|\right]^p \,dx \right)^{\frac{1}{p}} + c  \left(\mint_{\Omega_{4 r}}   \left|F \right|^{q} \, dx\right)^{\frac{1}{q}},
\end{eqnarray}
which proves the estimate \eqref{mainlocalest}.

We next assume that $q\in[p\gamma,\gamma^*(p-1))$ does not satisfies \eqref{case1}, that is,
$$
p \gamma \leq q\leq \max\left\{p,\frac{n(p-1)}{n-1}\right\}.
$$
Note that this happens only for the case that $p>n$ and $1<\gamma\leq \frac{n(p-1)}{p(n-1)}$, and that, in this case, we cannot find $\tilde q\in (1,\gamma)$ satisfying \eqref{tq}. Instead, let us set
$$
\tilde q:= \frac{1+\gamma}{2} \in (1,\gamma).
$$
Then, in the same argument above, we have the estimate \eqref{mainpf1}. Using this, Theorem~\ref{thmDwbdd1} (instead of Theorem \ref{thmDwbdd}) and H\"older's inequality, we obtain the estimate \eqref{mainpf2}. Hence, \eqref{mainlocalest} holds for the remaining range for $q$. This completes the proof.

\subsection{Proof of Corollary \ref{maincor}}\ \\
We take the test function $\varphi=u$ in the weak formulation \eqref{weakform}, and then use Young's inequality to arrive at
\begin{eqnarray*}
&&c(p,\nu)\int_{\Omega} |Du|^{p}\,dx + \int_{\Omega} V|u|^p \, dx \\
&&\quad \leq  \int_{\Omega} \mathbf{a}(x, Du) \cdot Du \, dx  + \int_{\Omega} V |u|^{p-2}u \cdot u  \, dx = \int_{\Omega} |F|^{p-2} F \cdot D u\, dx \\
&&\quad \leq   \int_{\Omega} |F|^{p-1} |Du| \,dx \leq c(\tau) \int_{\Omega} |F|^{p} \, dx + \tau \int_{\Omega} |Du|^{p}  \,dx
\end{eqnarray*}
for any small $\tau>0.$  Here we have used the inequality that $\mathbf{a}(x, \xi)\cdot \xi \geq c(p,\nu)  |\xi|^p,$ which can be easily obtained from \eqref{aas2}. We choose $\tau>0$ so small that
\begin{equation}\label{energyest}
\int_{\Omega} |Du|^{p}\,dx + \int_{\Omega} V|u|^p \, dx \leq  c \int_{\Omega} |F|^{p} \, dx. \end{equation}

On the other hand,  from the resulting estimates \eqref{mainlocalest} with $r = \frac{\tilde{R}}{4}$ where $\tilde{R} := \min\Big\{ R, \Vert V \Vert_{L^{\gamma}(\Omega)}^{-\frac{1}{p-\frac{n}{\gamma}}}\Big\}$  and $x_0 \in \overline{\Omega},$  we get that \begin{eqnarray}\label{Duestcov}
\nonumber&&\int_{\Omega_{\tilde{R}/4}(x_0)} |Du|^{q} + \chi_{\{q<p\gamma\}} \left[ V^{\frac1p} |u| \right]^q  \, dx  \\
&&\qquad \leq  \frac{c}{\tilde{R}^{n\left(\frac{q}{p}-1\right)}}\left( \int_{\Omega_{\tilde{R}}(x_0)} |Du|^p + \left[ V^{\frac1p} |u| \right]^p \,dx \right)^{\frac{q}{p}} + c  \int_{\Omega_{\tilde{R}}(x_0)}   \left|F \right|^{q} \, dx.
\end{eqnarray}
Since $\overline{\Omega}$ is compact, by Vitali's covering lemma, there exist finitely many points $x_0^1, \cdots, x_0^N$ in $\overline{\Omega}$ such that $B_{\tilde{R}/20}(x_0^k)$, $k=1,2,\dots,N$  are mutually disjoint and $\Omega \subseteq \bigcup_{k=1}^{N} B_{\tilde{R}/4}(x_0^k).$ Here we note that $\sum_{k=1}^N\chi_{B_{\tilde{R}}(x_0^k)}\leq c(n)$. Therefore from \eqref{Duestcov}, we deduce that
\begin{eqnarray}\label{glDuestcov}
\nonumber&&\int_{\Omega} |Du|^{q} + \chi_{\{q<p\gamma\}} \left[ V^{\frac1p} |u| \right]^q  \, dx \\
\nonumber&&\quad  \leq  \sum_{k=1}^{N} \int_{\Omega_{\tilde{R}/4}(x_0^k)} |Du|^{q} + \chi_{\{q<p\gamma\}} \left[ V^{\frac1p} |u| \right]^q  \, dx    \\
\nonumber&&\quad\leq c\,\sum_{k=1}^{N} \left\{\frac{1}{\tilde{R}^{n\left(\frac{q}{p}-1\right)}}\left( \int_{\Omega_{\tilde{R}}(x_0^k)} |Du|^p + \left[ V^{\frac1p} |u| \right]^p \,dx \right)^{\frac{q}{p}} +   \int_{\Omega_{\tilde{R}}(x_0^k)}   \left|F \right|^{q} \, dx\right\}\\
&&\quad \leq \frac{c}{\tilde{R}^{n\left(\frac{q}{p}-1\right)}}  \left( \int_{\Omega} |Du|^p + \left[ V^{\frac1p} |u| \right]^p \,dx \right)^{\frac{q}{p}} + c  \int_{\Omega}   \left|F \right|^{q} \, dx.
\end{eqnarray}
In turn, inserting \eqref{energyest} into \eqref{glDuestcov} and using H\"older's inequality  together with the fact that $\mathrm{diam}(\Omega)> \tilde{R}$, we obtain
\begin{eqnarray}
\nonumber&&\int_{\Omega} |Du|^{q} + \chi_{\{q<p\gamma\}} \left[ V^{\frac1p} |u| \right]^q  \, dx \leq  c\left(\frac{\mathrm{diam}(\Omega)}{\tilde{R}}\right)^{\frac{n(q-p)}{p}}\int_{\Omega}   \left|F \right|^{q} \, dx,
\end{eqnarray}
which implies the desired estimates \eqref{maingloest}.

%\section*{Acknowledgements}

\bibliographystyle{amsplain}

\begin{thebibliography}{10}


\bibitem{APT1} Ablowitz, M. J., Prinari, B. and Trubatch, A. D., \textit{Discrete and continuous nonlinear Schr\"odinger systems}, London Mathematical Society Lecture Note Series, 302. Cambridge University Press, Cambridge, 2004.


\bibitem{AM0} Acerbi, E. and Mingione, G.,\textit{Gradient estimates for the $p(x)$-Laplacean system}, J. Reine Angew. Math. \textbf{584} (2005), 117--148.

\bibitem{AM1} Acerbi, E. and Mingione, G.,
\textit{Gradient estimates for a class of parabolic systems},
Duke Math. J., \textbf{136} (2007), 285--320.



\bibitem{BS1} Berezin, F.A. and Shubin, M.A.,
\textit{The Schr\"odinger Equation},
Mathematics and Its Applications (Soviet Series), vol. 66, Kluwer Academic Publishers Group, Dordrecht, 1991.


\bibitem{BHS1} Bongioanni, B., Harboure, E. and Salinas, O.,\textit{Commutators of Riesz transforms related to Schr\"odinger operators},  J. Fourier Anal. Appl. \textbf{17} (2011), no. 1, 115--134.


\bibitem{BBHV1} Bramanti, M., Brandolini, L., Harboure, E. and Viviani, B., \textit{Global $W^{2,p}$ estimates for nondivergence elliptic operators with potentials satisfying a reverse H\"older condition}, Ann. Mat. Pura Appl. (4) \textbf{191} (2012), no. 2, 339--362.




\bibitem{BO1} Byun, S. and Ok, J., \textit{On $W^{1,q(\cdot)}$-estimates for elliptic equations of $p(x)$-Laplacian type}, J. Math. Pures Appl. (9) \textbf{106} (2016), no. 3, 512--545.



\bibitem{BOR1} Byun, S., Ok, J. and Ryu, S., \textit{Global gradient estimates for elliptic equations of $p(x)$-Laplacian type with BMO nonlinearity}, J. Reine Angew. Math. \textbf{715} (2016), 1--38.

\bibitem{BR1} Byun, S. and Ryu, S., \textit{Global weighted estimates for the gradient of solutions to nonlinear elliptic equations}, Ann. Inst. H. Poincar\'e Anal. Non Lin\'eaire \textbf{30} (2013), no. 2, 291--313.

\bibitem{BW1} Byun, S. and Wang, L., \textit{Elliptic equations with BMO coefficients in Reifenberg domains}, Comm. Pure Appl. Math.  \textbf{57} (2004), no. 10, 1283--1310.


\bibitem{CP1}Caffarelli L. A.  and Peral I., \textit{On $W^{1,p}$ estimates for elliptic equations in divergence form}, Comm. Pure Appl. Math. \textbf{51} (1998), no. 1, 1--21.



\bibitem{CFG1} Chiarenza, F., Fabes, E. and  Garofalo, N., \textit{Harnack's inequality for Schr\"odinger operators and the continuity of solutions}, Proc. Amer. Math. Soc.  \textbf{98} (1986), no. 3, 415--425.



\bibitem{CM1}Colombo, M. and Mingione, G., \textit{Calder\'on-Zygmund estimates and non-uniformly elliptic operators}, J. Funct. Anal. \textbf{270} (2016), no. 4, 1416--1478.


\bibitem{DM1}DiBenedetto, E. and Manfredi, J., \textit{On the higher integrability of the gradient of weak solutions of certain degenerate elliptic systems}, Amer. J. Math. \textbf{115} (1993), no. 5, 1107--1134.


\bibitem{D1}Di Fazio, G., \textit{H\"older-continuity of solutions for some Schr\"odinger equations}, Rend. Sem. Mat. Univ. Padova  \textbf{79} (1988), 173--183.


\bibitem{Fe} Fefferman, C., \textit{The uncertainty principle}, Bull. Amer. Math. Soc. (N.S.)  \textbf{9} (1983), no. 2, 129--206.


\bibitem{FPR} Filippucci, R., Pucci, P. and R\u adulescu, V.,  \textit{Existence and non-existence results for quasilinear elliptic exterior problems with nonlinear boundary conditions}, Comm. Partial Differential Equation  \textbf{33} (2008), no. 4--6, 706--717.


\bibitem{Ge} Gehring, F.W., \textit{The $L^p$-integrability of the partial derivatives of a quasiconformal mapping}, Acta Math.  \textbf{130} (1973), 265--277.

\bibitem{Gi} Giusti, E., \textit{Direct methods in the calculus of variations},  World Scientific Publishing Co., Inc., River Edge, NJ, 2003.


\bibitem{G1} Grafakos, L., \textit{Modern Fourier analysis},  Second edition. Graduate Texts in Mathematics, 250. Springer, New York, 2009.


\bibitem{HL1} Han, Q. and Lin, F., \textit{Elliptic Partial Differential Equations},  Courant Lecture Notes in Mathematics, 1, New York University, Courant Institute of Mathematical Sciences, New York, American Mathematical Society, Providence, RI, 1997.



\bibitem{Iw1} Iwaniec, T., \text{Projections onto gradient fields and $L^p$-estimates for degenerated elliptic operators}, Studia Math. \textbf{75} (1983), no. 3, 293--312.



\bibitem{KZ1} Kinnunen, J. and Zhou, S.,
\textit{A local estimate for nonlinear equations with discontinuous coefficients},
Comm. Partial Differential Equations \textbf{24} (1999), no. 11-12, 2043--2068.


\bibitem{K1} Kurata, K.,
\textit{Continuity and Harnack's inequality for solutions of elliptic partial differential equations of second order},
Indiana Univ. Math. J. \textbf{43} (1994), no. 2, 411--440.


\bibitem{LU} Ladyzhenskaya, O. and Ural'tseva, N. N., \textit{Linear and quasilinear elliptic equations}, Math. Sci. Eng., vol. 46, Academic Press, New York, 1968.


\bibitem{LO1} Lee, M. and Ok, J.
\textit{Nonlinear Calder\'on-Zygmund theory involving dual data}, Revista Matem\'atica Iberoamericana, \textbf{35} (2019), no. 4, 1053--1078.

\bibitem{LS1} Litvak, A.G. and Sergeev, A.M., \textit{One dimensional collapse of plasma waves}, JETP Letters  \textbf{27} (1978), 517--520.



\bibitem{MF1} Makhankov, V.G. and Fedyanin, V.K., \textit{Non-linear effects in quasi-one-dimensional models of condensed matter theory}, Phys. Rep.  \textbf{104} (1984), 1--86.



\bibitem{Mu} Muckenhoupt, B., \textit{Weighted norm inequalities for the Hardy maximal function}, Trans. Amer. Math. Soc.  \textbf{165} (1972),  207--226.







\bibitem{MP1} Mengesha, T. and Phuc, N. C., \textit{Global estimates for quasilinear elliptic equations on Reifenberg flat domains},  Arch. Ration. Mech. Anal. \textbf{203} (2012), no. 1, 189--216.

\bibitem{Min1} Mingione, G., \textit{The Calder\'on-Zygmund theory for elliptic problems with measure data}, Ann. Sc. Norm. Super. Pisa Cl. Sci. (5) \textbf{6} (2007), no. 2, 195--261.

\bibitem{Mis1} Misawa, M., \textit{$L^q$-estimates of gradients for evolutional $p$-Laplacian systems},
J. Differential Equations \textbf{219} (2005), no. 2, 390--420.







\bibitem{Ok1} Ok, J., \textit{Gradient estimates for elliptic equations with $L^{p(\cdot)}\log L$ growth}, Calc. Var. Partial Differential Equations \textbf{55} (2016), no. 2, Art. 26, 30 pp.


\bibitem{PS1} Palagachev, D.K. and Softova, L., \textit{The Calder\'on-Zygmund property for quasilinear divergence form equations over Reifenberg flat domains}, Nonlinear Anal. \textbf{74} (2011) no. 5, 1721--1730.




\bibitem{PT1} Pan, G. and Tang, L., \textit{Solvability for Schr\"odinger equations with discontinuous coefficients}, J. Funct. Anal. \textbf{270} (2016) no. 1, 88--133.


\bibitem{Re1} Reifenberg, E.R., \textit{Solution of the Plateau Problem for $m$-dimensional surfaces of varying topological type}
Acta Math. \textbf{104} (1960) 1--92.

\bibitem{Ph1} Phuc, N. C., \textit{Nonlinear Muckenhoupt-Wheeden type bounds on Reifenberg flat domains, with applications to quasilinear Riccati type equations}, Adv. Math. \textbf{250} (2014), 387--419.




\bibitem{Sh0} Shen, Z., \textit{On the Neumann problem for Schr\"odinger operators in Lipschitz domains}, Indiana Univ. Math. J. \textbf{43} (1994), no. 1, 143--176.


\bibitem{Sh1} Shen, Z., \textit{$L^p$ estimates for Schr\"odinger operators with certain potentials}, Ann. Inst. Fourier  \textbf{45} (1995), no. 2, 513--546.

\bibitem{Sho1} Showalter, R. E., 
\textit{Monotone operators in Banach Space and Nonlinear Partial Differential Equations},
Mathematical Surveys and Monographs, Vol. 49, American Mathematical Society, Providence, RI, 1997.


\bibitem{SS} Sulem, C. and Sulem, P.L., \textit{The nonlinear Schr\"odinger equation: self-focusing and wave collapse}, Vol. 139. Springer Science \& Business Media, 2007.


\bibitem{To1} Toro, T., \textit{Doubling and flatness: geometry of measures}, Notices Amer. Math. Soc. \textbf{44} (1997), no. 9, 1087--1094.

\bibitem{Um1} Um, K., \textit{Elliptic equations with singular BMO coefficients in Reifenberg domains}, J. Differential Equations \textbf{253} (2012), no. 11, 2993--3015.



\end{thebibliography}

\end{document}